\documentclass[11pt]{amsart}
\usepackage{hyperref,amsmath,a4wide,graphicx,wasysym,amssymb,mathtools,xcolor}
\usepackage{tikz}
\usetikzlibrary{arrows,calc,shapes,decorations.pathreplacing,cd,patterns}
\usepackage[T1]{fontenc}
\usepackage[latin1]{inputenc}
\usepackage{hyperref}
\usepackage[hyperpageref]{backref}
\numberwithin{equation}{section}
\newtheorem{lemma}[equation]{Lemma}
\newtheorem{proposition}[equation]{Proposition}
\newtheorem{theorem}[equation]{Theorem}
\usepackage[all]{xy}
\theoremstyle{definition}
\newtheorem{definition}[equation]{Definition}
\newtheorem{example}[equation]{Example}
\newtheorem{corollary}[equation]{Corollary}
\newtheorem{remark}[equation]{Remark}

\newcommand{\mb}[1]{{\mathbf #1}}
\newcommand{\mc}[1]{{\mathcal #1}}

\newcommand{\norm}[1]{\left\lVert#1\right\rVert}
\DeclareMathOperator*{\hocolim}{hocolim}
\DeclareMathOperator*{\colim}{colim}
\title{Strictifying and taming directed paths in Higher Dimensional Automata}
\author{Martin Raussen} 
\address{Department of
  Mathematical Sciences, Aalborg University, Skjernvej 4A,
  DK-9220 Aalborg {\O}st, Denmark} 
\email{raussen@math.aau.dk}
\thanks{The author thanks Uli Fahrenberg (\'{E}cole Polytechnique, Paris) and Krzysztof Ziemia\'{n}ski (Warsaw) for helpful conversations; Ziemia\'{n}ski particularly for pointing out several uncorrect statements in previous versions. Thanks are also due to the anonymous referees for several hints leading to improvements of the presentation.}
\begin{document}
\begin{abstract}
Directed paths have been used by several authors to describe concurrent executions of a program. Spaces of directed paths in an appropriate state space contain executions with all possible legal schedulings. It is interesting to investigate whether one obtains different topological properties of such a space of executions if one restricts attention to schedulings with ``nice'' properties, eg involving synchronizations. This note shows that this is not the case, ie that one may operate with nice schedulings without inflicting any harm.

Several of the results in this note had previously been obtained by Ziemia\'{n}ski in \cite{Ziemianski:17,Ziemianski:19}. We attempt to make them accessible for a wider audience by giving an easier proof for these findings by an application of quite elementary results from algebraic topology; notably the nerve lemma.
\end{abstract}
\maketitle

\section{Introduction} 
\subsection{Schedules in Higher Dimensional Automata} Higher Dimensional Automata (HDA) were introduced by V.\ Pratt \cite{Pratt:91} back in 1991 as a model for concurrent computation. Mathematically, HDA can be described as (labelled)
pre-cubical or $\Box$-sets (cf Definition \ref{def:d}). Those are obtained by glueing individual cubes of various dimensions together; 
directed paths corresponding to a $\Box$-set respect the natural partial
order in each cube of the model. These directed paths correspond to lawful schedules/executions of a concurrent computation; and paths that are homotopic in a directed sense (d-homotopic, cf \cite{FGHMR:16}), will always lead to the same result.

Compared to other well-studied concurrency models like labelled
transition systems, event structures, Petri nets etc., it has been shown by R.J.~van
Glabbeek \cite{Glabbeek:06} that Higher Dimensional Automata have the
highest expressivity (up to history-preserving bisimilarity); on the other hand, they are certainly less
studied and less often applied so far. Recently, even more general partial Higher Dimensional Automata have been proposed and studied \cite{FL:15,Dubut:19}. 

It is not evident which paths one should admit as directed paths: It is obvious that they should progress along each axis in each facet of the HDA; the time flow is not reversible. This is reflected in the notion of a \emph{d-path} on such a complex. One may ask, moreover, that processes synchronize after a step (either a full step or an idle step) has been taken. This is what \emph{tame d-paths} have to satisfy, on top. A natural question to ask is whether one can perform the same computations (and obtain the same results) according to whether synchronization is requested all along or not.

It has been shown by K.\ Ziemia\'{n}ski \cite{Ziemianski:17, Ziemianski:19} that the synchronization request has no essential significance: The spaces of directed paths and of tame d-paths between two states are always \emph{homotopy equivalent}. This has two consequences: On the one hand, one may, without global effects, relax the computational model and allow quite general parallel compositions. On the other hand, in the analysis of the schedules on a HDA, one may restrict attention to tame d-paths, ie mandatory synchronization; these are combinatorially far easier to model and to analyze. 

\subsection{Posets, poset categories, and algebraic topology}
Many (sets of) schedules can be formulated in the language of series-parallel pomsets (partially ordered multisets of events). Tame d-paths ``live in'' \emph{serial} compositions of simple Higher Dimensional Automata consisting of a single cube each. General d-paths underpin more complicated schedules, for which \emph{parallel} composition is involved in the description; cf eg \cite{FM:09}  for a detailed description of finite step transition systems accepting pomset languages and \cite{FJST:19} for newer developments.

In this paper, we are not interested in the analysis of individual paths/schedules, but in the analysis of the space of \emph{all} schedules from a start state to an end state, equipped with a natural topology. It turns out that the way subspaces of schedules are glued together is essentially the same, whether synchronization is mandatory or not.

In that line of argument, posets enter the scene in a different manner: We divide the space of all executions (d-paths) into easy-to-analyze subspaces; for tame d-paths, for example, we simply fix a sequence of faces that they 
are kept in. Refinement is a partial order relation on these face sequences, and we will exploit the combinatorial/topological properties of the poset category of face sequences (called cube chains \cite{Ziemianski:17, Ziemianski:19}).

The use of methods from algebraic topology in the analysis of concurrency properties has been advocated in eg \cite{FGR:06,Grandis:09,FGHMR:16} to which we refer the reader for details. In this paper, we will (apart from the proof of Proposition \ref{prop:general}) only apply one important result from algebraic topology, the so-called nerve lemma, cf Theorem \ref{thm:nl}. At a first glance, one may say that it allows to apply a divide and conquer strategy: Cut a space into subspaces that are topologically trivial (contractible); also all non-empty intersections of such are assumed contractible. Then all essential information (up to homotopy equivalence) is contained in the way these subspaces are glued together. That glueing can be described by way of a simplicial complex, the \emph{nerve} of the associated poset category. If the posets associated to different spaces are (naturally) isomorphic, then their nerves and hence the spaces they describe are homotopy equivalent.
 
\section{Definitions and results}
\subsection{Definitions} 
We start with some notation: The unit interval $[0,1]$ is denoted by $I$. For two topological spaces $X$ and $Y$, we let $Y^X$ denote the space of all continuous maps from $X$ to $Y$ equipped with the compact open-topology. For an interval $J=[a,b]\subset\mb{R},\; a<b$, an element $p\in X^J$ is called 
a \emph{path} in $X$. A path $\varphi\in J_1^{J_2}$ in an interval $J_1$ defines a \emph{reparametrization map} $X^{J_1}\to X^{J_2}, p\mapsto p\circ\varphi$.

Let $p_0:[t_0,t_1]\to X$ and $p_1:[t_1,t_2]\to X$ denote two paths with $p_0(t_1)=p_1(t_1)$. Their \emph{concatenation} at $t_1$ is denoted $p_0*_{t_1}p_1:[t_0,t_2]\to X$.
\begin{definition}\label{def:d}
\begin{enumerate}
\item \cite{Grandis:01,Grandis:09} A \emph{d-space} consists of a topological space $X$ together with a subspace $\vec{P}(X)\subset X^I$ of paths in $X$ that contains the constant paths, is closed under concatenation and under the action of the monoid\\ $\vec{P}(I):=\{ p: I\to I,\; t\le t'\Rightarrow p(t)\le p(t')\}$ of \emph{increasing = non-decreasing} reparametrizations $p: I\to I$ under composition. Elements of $\vec{P}(X)$ are called \emph{d-paths}.
\item For $x,y\in X$, we let $\vec{P}(X)_x^y=\{ p\in\vec{P}(X)|\; p(0)=x, p(1)=y\}$ denote the subspace of all d-paths from $x$ to  $y$.
\item A continuous  map $f:X\to Y$ is called a \emph{directed} map if $f(\vec{P}(X))\subset\vec{P}(Y)$, ie if it maps d-paths in $X$ into d-paths in $Y$.
  \item Let $J=[a,b]\subset\mb{R}$ denote an interval ($a<b$) and let $\varphi: J\to I$ denote any \emph{increasing homeomorphism}. Then $\vec{P}_J(X):=\{ p\circ\varphi |\; p\in\vec{P}(X)\}$ -- independent of the choice of $\varphi$.
\end{enumerate}
\end{definition}

In applications to concurrency, the d-spaces under consideration are usually directed $\Box$-sets, or rather their geometric realizations \cite{Pratt:91, Glabbeek:91, Glabbeek:06, FGHMR:16}:
 
\begin{definition}\label{df:pcs}
\begin{enumerate}
\item $\Box^1 := (I,\vec{P}(I))$, cf Definition \ref{def:d}(1); $\Box^n:=(I^n,(\vec{P}(I)^n)$. Hence d-paths in $\Box^n$ have non-decreasing coordinate paths. 
\item A $\Box$\emph{-set} $X$ (also called a pre-cubical or semi-cubical set) is a sequence of disjoint sets $X_n,\; n\ge 0$; equipped, for $n>0$, with \emph{face maps}\\ $d_i^{\alpha}: X_n\to X_{n-1},\; \alpha\in\{0,1\}, 1\le i\le n$, satisfying the pre-cubical relations:\\ $d_i^{\alpha}d_j^{\beta}=d_{j-1}^{\beta}d_i^{\alpha}$ for $i<j$.\\
Elements of $X_n$ are called $n$-\emph{cubes}, those of $X_0$ are called \emph{vertices}.\\
Iterated face maps $d_1^0\circ\dots d_1^0: X_n\to X_0$, resp.\ $d_1^1\circ\dots d_1^1: X_n\to X_0$ project an $n$-cube $c$ to its \emph{source} vertex $c_{\mb{0}}$ and \emph{target} vertex $c_{\mb{1}}$.
\item  A $\Box$-set $X$ is called \emph{proper} \cite{Ziemianski:17} if for every pair $x_0, x_1\in X_0$ of vertices there exists at most one cube with bottom vertex $x_0$ and top vertex $x_1$.
\item A $\Box$-set is called \emph{non-self-linked} \cite{FGHMR:16} if every cube $c\in X_n$ has $\binom{n}{k}2^{n-k}$ different iterated faces in $X^k$ (ie, iterated faces agree if and only they do so because of the pre-cubical relations).
\item The \emph{geometric realization} of a pre-cubical set $X$ is the space
\[|X|=\bigcup_{n\ge 0}X_n\times I^n/_{[d_i^{\alpha}(c),x]\sim [c,\delta_i^{\alpha}(x)]}\] 
with $\delta_i^{\alpha}(x_1,\dots ,x_{n-1})=(x_1,\dots, x_{i-1}, \alpha, x_i,\dots , x_{n-1})$.
\end{enumerate}
\end{definition}

As far at the author is aware of, Higher-Dimensional Automata in the concurrency literature are often proper and non-self-linked. 

Speaking about a cube $c$ in $X$ (or rather in $|X|$; we will often suppress $||$ from the notation), we mean actually the image of the quotient map $\{ c\}\times I^{\dim c}\hookrightarrow\bigcup_{n\ge 0}X_n\times I^n\downarrow |X|$. If $X$ is non-self-linked, then this map is a homeomorphism onto its image in $X$; if not, then it may identify points on the boundary of $I^{\dim c}$. 

What are the \emph{directed} paths in the geometric realization of a $\Box$-set?
 \begin{definition}\label{def:present}
    \begin{enumerate}
    \item A path $p:J\to I$ on an interval $J$  is called \emph{strictly} increasing if it is increasing and moreover: $p(t)=p(t')\Rightarrow t=t'$ or $p(t)=0$ or $p(t)=1$.
    \item A path $p=(p_1,\dots ,p_n): J\to I^n$ on an interval $J$ is called
      \emph{(strictly) increasing} if every component $p_i$ is (strictly)
      increasing. 
    \item  Let $X$ denote a $\Box$-set. A path $p\in X^I$ is called a \emph{d-path} if it admits a presentation \cite[2.6]{Ziemianski:19} $[c_1;\beta_1]*_{t_1}[c_2;\beta_2]*_{t_2}\dots *_{t_{l-1}}[c_l;\beta_l]$ consisting of 
    \begin{itemize}
    \item a sequence of real numbers  $0=t_0\le t_1\le\dots t_{i-1}\le t_i\le\dots\le t_l=1$,
 \item a sequence $(c_i)$ of cubes in $X$, 
 \item a sequence $(\beta_i)\in\vec{P}_{[t_{i-1},t_i]}(I^{\dim c_i}),\; 1\le i\le l,\;$ of \emph{increasing} paths $\beta_i$
 \end{itemize} such that $p(t)=[c_i;\beta_i(t)], t_{i-1}\le t\le t_i$; i.e., on this interval, $p=q_i\circ\beta_i$ with $q_i: I^{\dim c_i}\to c_i$ the resp.\ quotient map.
    
  \item A d-path
    $p:I\to X$ is called \emph{strictly} directed if there exists
    a presentation $p=[c_1;\beta_1]*_{t_1}[c_2;\beta_2]*_{t_2}\dots *_{t_{l-1}}[c_l;\beta_l]$ with \emph{strictly} increasing paths $\beta_i: [t_{i-1},t_i]\to I^{\dim c_i}$.
\item A directed path $p:I\to X$ is called \emph{tame} if the
  subdivision in (3) above can be chosen such that $p(t_i)$ is a \emph{vertex} for
  every $0\le i\le l$. A path that is strictly directed and tame is
  called \emph{strictly tame}.
 \end{enumerate}  \end{definition}
 Observe that we allow d-paths that include non-trivial directed \emph{loops}.
 
\begin{figure}[h]
\begin{tikzpicture}
\draw (0,0) -- (2,0) -- (2,4) -- (0,4) -- (0,0);
\draw (0,2) -- (2,2);
\draw (3,0) -- (5,0) -- (5,4) -- (3,4) -- (3,0);
\draw (3,2) -- (5,2);
\draw (6,0) -- (8,0) -- (8,4) -- (6,4) -- (6,0);
\draw (6,2) -- (8,2);
\draw (9,0) -- (11,0) -- (11,4) -- (9,4) -- (9,0);
\draw (9,2) -- (11,2);
\draw[->, line width=0.6mm, color=blue] (0,0) -- (1,1); 
\draw[->, line width=0.6mm, color=blue] (1,1) -- (1,3);
\draw[->, line width=0.6mm, color=blue] (1,3) -- (2,4);
\draw[->, line width=0.6mm, color=blue] (3,0) -- (3.8,1);
\draw[->, line width=0.6mm, color=blue] (3.8,1) -- (4.2,3);
\draw[->, line width=0.6mm, color=blue] (4.2,3) -- (5,4);
\draw[->, line width=0.6mm, color=blue] (6,0) -- (7,1);
\draw[->, line width=0.6mm, color=blue] (7,1) -- (7,1.5);
\draw[->, line width=0.6mm, color=blue] (7,1.5) -- (8,2);
\draw[->, line width=0.6mm, color=blue] (8,2) -- (8,4);
\draw[->, line width=0.6mm, color=blue] (9,0) -- (9.8,1);
\draw[->, line width=0.6mm, color=blue] (9.8,1) --(11,2);
\draw[->, line width=0.6mm, color=blue] (11,2) -- (11,4);
\end{tikzpicture}
\caption{d-paths in a cubical complex consisting of two squares: directed, strict, tame, tame and strict.}
\end{figure}
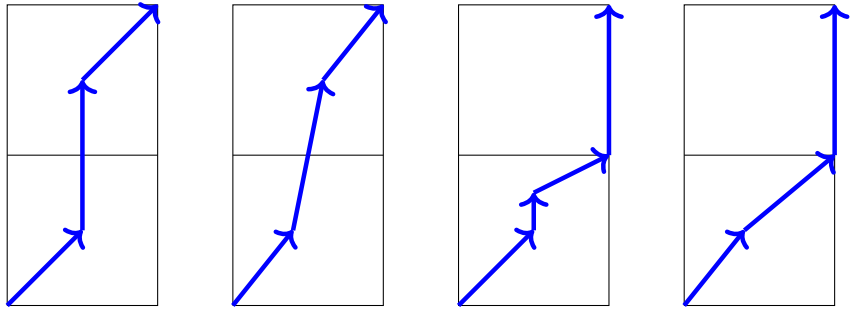

\begin{example}\label{ex:cubset}
\begin{enumerate}
\item A Euclidean cubical complex \cite{RZ:14} $K$ is a subset of Euclidean space $\mb{R}^n$ that is a union of elementary cubes $\prod [k_i,k_i+e_i]\subset\mb{R}^n$ with $k_i\in\mb{Z}$ and $e_i\in\{0 , 1\}$. The maximal cubes in the cubical set that it realizes can be described by a pair of bottom and top vertices $(\mb{k},\mb{l})\in\mb{Z}^n\times\mb{Z}^n$ with $0\le l_i-k_i\le 1$. A Euclidean cubical complex is obviously proper and non-self-linked. Euclidean cubical complexes arise as models for $PV$-programs (cf eg \cite{FGHMR:16}).
\item For an example of a non-proper cubical set, consider the cubical set $X$ glued from two squares (2-cubes) along a common boundary (consisting of four oriented edges and of four vertices). Its geometric realization is homeomorphic to a 2-dimensional sphere; it is less obvious how to describe the directed paths on this sphere via the homeomorphism. The space of all directed paths from the common source to the common target of both squares is actually homotopy equivalent to a \emph{circle} that may be represented by d-paths through the \emph{union of the diagonals} of the two squares.
\item \cite{Ziemianski:20} The $\Box$-set $Z_n$ with exactly one cube in every dimension $k\le n$ is obviously not proper and self-linked. For a description of d-paths in the geometric realization of this space, cf Example \ref{ex:onecube}.
\item \cite{Ziemianski:20} The $\Box$-set $Q^n$ has $(n-k+1)$ $k$-cubes $c_0^k,\dots , c_{n-k}^k$ and face maps $d^{\alpha}_ic^k_j=c^{k-1}_{j+i+\alpha}$. It arises from the cube $I^n$ by identifying all faces spanned by two vertices with $i$, resp.\ $k+i$ coordinates $1$ with each other ($0\le i\le n-k$). This $\Box$-set is proper, but also self-linked.
\end{enumerate}
\end{example}

\newpage\subsection{Interpretation}
\subsubsection{Different types of schedulings}

\begin{description}
\item[D-(irected)] paths correspond to executions of a concurrent program
  - without the possibility to let one or several processes run
  backwards in time.
\item[Strict] d-paths correspond to programs where a particular process only may
  be idle at a vertex in the program (once a step is fully taken); between steps it needs to move forward in
  time at ``positive speed''.
\item[Tame] d-paths correspond to programs where processes need to
  synchronize after every step before progressing. A number of processes may stay
  idle inbetween. Hence at synchronization events, a process has taken a full step or it has stayed idle. 
  \item[Strict tame] d-paths correspond to programs combining both properties.
\end{description}

Our main result in Theorem \ref{thm} below states that the spaces of schedulings, regardless of the restrictions above, will have the same topological
properties in all four cases.

In the final Section \ref{ss:kink} we show that one may restrict  the space of tame d-paths even further, up to homotopy equivalence: It is enough to consider PL d-paths that are piecewise linear with kink points at certain hyperplanes.

\subsubsection{A simple illustrative example}\label{sss:ex}
We refer to Figure \ref{fig:hex}. Let $X=\partial I^3$ be the $\Box$-set corresponding to the boundary of a 3-cube. It has twelve edges: four parallel to each of the axes and labelled $x, y$ resp.\ $z$ and six two-dimensional facets: two parallel to each of the three coordinate planes and labelled $xy, xz$ resp.\ $yz$.

The image of every d-path from the bottom vertex $\mb{0}$ to the top vertex $\mb{1}$ is contained in two subsequent square facets; the image of every \emph{tame} d-path from $\mb{0}$ to $\mb{1}$ is contained in a pair of an edge and a facet. Taking care of intersections, one arrives in both cases at a category with geometric realization in form of a hexagon; homotopy equivalent to the circle $S^1$. The space of all d-paths (whether tame or not) in $\partial I^n$ is indeed homotopy equivalent to $S^{n-2}$. 

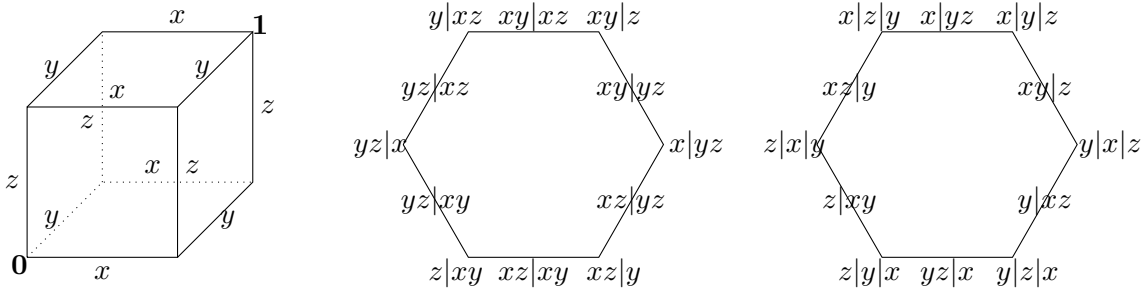
\begin{figure}[h]
\begin{center}
\begin{tikzpicture}
\draw (0,0) -- (2,0) -- (2,2) -- (0,2) -- (0,0);
\draw (2,0) -- (3,1) -- (3,3) -- (2,2);
\draw (0,2) -- (1,3) -- (3,3);
\draw [dotted] (0,0) -- (1,1) -- (3,1);
\draw [dotted] (1,1) -- (1,3);
\node at (1,-0.2) {$x$};
\node at (1.2,2.2) {$x$};
\node at (1.67,1.2) {$x$};
\node at (2,3.2) {$x$};
\node at (0.33,0.5) {$y$};
\node at (2.67,0.5) {$y$};
\node at (0.33,2.5) {$y$};
\node at (2.33,2.5) {$y$};
\node at (-0.2,1) {$z$};
\node at (2.2,1.2) {$z$};
\node at (0.8,1.8) {$z$};
\node at (3.2,2) {$z$};
\node at (-0.1,-0.1) {$\mb{0}$};
\node at (3.1,3.1) {$\mb{1}$};
\draw (5,1.5) -- (5.866,0) -- (7.6,0) -- (8.46,1.5) -- (7.6,3) -- (5.866,3) -- (5,1.5);
\draw (10.5,1.5) -- (11.366,0) -- (13.1,0) -- (13.96,1.5) -- (13.1,3) -- (11.366,3) -- (10.5,1.5);
\node at (6.733,3.2) {$xy|xz$};
\node at (5.433,2.25) {$yz|xz$};
\node at (5.433,0.75) {$yz|xy$};
\node at (6.733,-0.2) {$xz|xy$};
\node at (8.03,0.75) {$xz|yz$};
\node at (8.03,2.25) {$xy|yz$};
\node at (4.7,1.5){$yz|x$};
\node at (8.9,1.5){$x|yz$};
\node at (5.7,3.2){$y|xz$};
\node at (7.8,3.2){$xy|z$};
\node at (5.7,-0.2){$z|xy$};
\node at (7.8,-0.2){$xz|y$};
\node at (12.233,3.2) {$x|yz$};
\node at (10.933,2.25) {$xz|y$};
\node at (10.933,0.75) {$z|xy$};
\node at (12.233,-0.2) {$yz|x$};
\node at (13.53,0.75) {$y|xz$};
\node at (13.53,2.25) {$xy|z$};
\node at (10.2,1.5){$z|x|y$};
\node at (14.4,1.5){$y|x|z$};
\node at (11.2,3.2){$x|z|y$};
\node at (13.3,3.2){$x|y|z$};
\node at (11.2,-0.2){$z|y|x$};
\node at (13.3,-0.2){$y|z|x$};
\end{tikzpicture}
\end{center}
\caption{$X=\partial I^3$, and spaces of d-paths, resp.\ of tame d-paths}\label{fig:hex}
\end{figure}

The $\Box$-set $X=\partial I^3$ models the situation where a shared resource can serve two out of three processes but not all of them at the same time. Remark that a sequence like $xy|xz$ (on top of Figure \ref{fig:hex}) of two subsequent facets can be interpreted as $x\parallel (y|z)$, ie $x$ and $y|z$ are executed concurrently. Allowing this may be very convenient and speed up a concurrent execution. For an analysis of the consequences, it is reassuring to realize that the space of schedules between two given states is qualitatively the same regardless whether one allows parallel execution over a series of steps (like in $x\parallel (y|z)$) or only over one step at a time (like in $yz$).  

\subsection{Main result} Let $X$ denote a $\Box$-set with finitely many cubes.
For a given pair of vertices $x^-, x^+\in X_0$, we let
$\vec{P}(X)_{x^-}^{x^+}, \vec{S}(X)_{x^-}^{x^+},
\vec{T}(X)_{x^-}^{x^+},$ resp.\ $\vec{ST}(X)_{x^-}^{x^+}$
denote the spaces of directed, strictly directed, tame, and strictly
tame dipaths from $x^-$ to $x^+$ (considered as subspaces of $X^I$
with the compact-open topology). Inclusion maps lead to the commutative diagram
\begin{equation}
  \label{eq:incl}
  \xymatrix{& |\mc{C}(X)_{x^-}^{x^+}| &\\
\vec{ST}(X)_{x^-}^{x^+}\ar[ur]^{\textcircled{5}}\ar@{^{(}->}[rr]^{\textcircled{3}}\ar@{^{(}->}[d]^{\textcircled{1}} & & \vec{S}(X)_{x^-}^{x^+}\ar@{^{(}->}[d]^{\textcircled{2}}\ar[ul]_{\textcircled{6}}\\
\vec{T}(X)_{x^-}^{x^+}\ar@{^{(}->}[rr]^{\textcircled{4}} & & \vec{P}(X)_{x^-}^{x^+}.}
\end{equation}
that also contains (maps into) the nerve of a poset-category $\mc{C}(X)_{x^-}^{x^+}$ explained in the sketch of the proof of our result:

\begin{theorem}\label{thm}
\begin{enumerate}
\item All inclusion maps in (\ref{eq:incl}) are homotopy equivalences. 
\item For a proper non-self-linked $\Box$-set $X$ (cf Definition \ref{df:pcs}(2)), all path spaces are homotopy equivalent to the nerve of the category $\mc{C}(X)_{x^-}^{x^+}$.
\end{enumerate}
\end{theorem}

\begin{proof}[Overview proof]
It is shown in Proposition \ref{prop:strict} by a cube-wise strictification construction that the maps with labels $\textcircled{1}$ and $\textcircled{2}$ are homotopy equivalences.

For a $\Box$-set $X$ and chosen end points $x^-, x^+$, we define a poset category $\mc{C}(X)_{x^-}^{x^+}$, cf Section \ref{ss:pathcol}: An object of that category is a \emph{cube chain} (cf Definition \ref{def:cc}) in $X$ connecting $x^-$ with $x^+$; this a sequence of cubes such that the top vertex of each cube in that sequence agrees with the bottom vertex of the subsequent cube. Morphisms in $\mc{C}(X)_{x^-}^{x^+}$ correspond then to \emph{refinements} of cube chains; for details consult Section \ref{ss:pathcol}.

We prove in Proposition \ref{prop:nerveC} for paths in a \emph{proper non-self-linked} $\Box$-set $X$ (cf Definition \ref{df:pcs}(2-3)) that both $\vec{S}(X)_{x^-}^{x^+}$ and $\vec{ST}(X)_{x^-}^{x^+}$ are homotopy equivalent to the nerve of $\mc{C}(X)_{x^-}^{x^+}$ (indicated by the maps $\textcircled{5}$ and $\textcircled{6}$) and can therefore deduce that also $\textcircled{3}$ is a homotopy equivalence: Both spaces have a common underlying combinatorial structure!  

Our proof uses only the classical nerve lemma, cf Theorem \ref{thm:nl}, and a transparent taming construction (Proposition \ref{prop:contr}) for strict d-paths subordinate to the \emph{collar} of a cube chain (cf Definition \ref{def:subo}). The remaining inclusion $\textcircled{4}$ is a homotopy equivalence as well by the 2-out-of-3 property for homotopy equivalences. In the more involved case of a \emph{general} $\Box$-set $X$, we show in Proposition \ref{prop:general} that $\textcircled{3}$ is a homotopy equivalence using the projection lemma and the homotopy lemma (cf e.g.\ \cite[Theorem 15.19]{Kozlov:08} and \cite[Theorem 15.12]{Kozlov:08}), both underlying the proof of the nerve lemma.
\end{proof}

\begin{remark} Many of the results in this paper are not new. Ziemianski proved in \cite{Ziemianski:19}, using elaborate homotopy theory tools, that the space of tame d-paths $\vec{T}(X)_{x^-}^{x^+}$ is always homotopy equivalent to the nerve of a more intricate category $Ch(X)$ of cube chains, even for a general $\Box$-set $X$. This \emph{Reedy category} (cf \cite[Definition 5.2.1]{Hovey:99}) takes care of identifications on the boundary of cubes in a cube chain. Moreover, he shows by an ingenious global taming construction, that $\textcircled{4}$ is a homotopy equivalence. Apart from including spaces of strictly increasing paths (necessary in our proof for taming), this note presents a far more elementary argument that, for proper $\Box$-sets, only uses the nerve lemma.
\end{remark}

\section{Strictification}\label{s:strict}
\subsection{Strictifying directed maps on $\Box$-sets}\label{ss:vf}
\begin{lemma}\label{lem:strict}
There exists a (continuous) directed map $F: \Box^2\to\Box^1$ (cf Definition \ref{def:d}(3) and \ref{df:pcs}(1)) with the following properties:
\begin{enumerate}
\item $x\in I\Rightarrow F(0,x)=x$. 
\item $ t\in I\Rightarrow F(t,0)=0, F(t,1)=1$.
\item $0<x<1,\; t\in I\Rightarrow 0<F(t,x)<1$.
\item $x<y,\; t\in I\Rightarrow F(t,x)<F(t,y)$.
\item $s, t\in I,\; s<t,\; 0<x<1\Rightarrow F(s,x)<F(t,x)$.
\end{enumerate} 
\end{lemma}

\begin{proof}
One way to construct such a directed map is as the restriction of the flow of the differential equation $y'=g(y)$ corresponding to a smooth function $g:I\to\mb{R}$ with $g(0)=g(1)=0$ and $g(t)>0,\; 0<t<1$, e.g. $g(t)=t-t^2$. The restriction of its flow, ie the function given by $F(t,x)=\frac{xe^t}{1-x+xe^t}$, has the required properties.
\end{proof}

We may interpret the map $F$ from Lemma \ref{lem:strict} as a homotopy of d-paths $F: I\times\Box^1\to\Box^1$ and use it to define a diagonal continuous directed homotopy $\mb{F}: I\times\Box^n\to\Box^n$ on the cube $\Box^n$ by $\mb{F}(t; x_1,\dots ,x_n)=(F(t,x_1),\dots ,F(t,x_n))$. Remark that $\mb{F}$ respects all (sub)-faces of $\Box^n$ because of Lemma \ref{lem:strict}(2). Applying this construction cube-wise (the same for every $k$-cube!), we define for every (geometric realization of a) semi-cubical set $X$, a continuous directed map $\mb{F}:I\times X\to X$ that lets all cubes -- and in particular all vertices -- invariant.

Using a directed map $\mb{F}$ as in Lemma \ref{lem:strict}, we define a strictifying map\\ $\vec{S}:\vec{P}(X)_{x^-}^{x^+}\to\vec{S}(X)_{x^-}^{x^+}$ by $\vec{S}(p)(t):=\mb{F}(t,p(t))$.\\ Start and end points $x^-$ and $x^+$ are vertices and therefore unchanged.

\begin{lemma}\label{lem:flow}
Let $p\in\vec{P}(X)_{x^-}^{x^+}$.
\begin{enumerate}
\item If $p(t)\in c$ for some cube $c$ in $X$, then $\vec{S}(p)(t)\in c$ for all $t\in I$.  
\item $\vec{S}(p)\in\vec{S}(X)_{x^-}^{x^+}$.
\item If $p$ is tame, then $\vec{S}(p)$ is (strict and) tame.
\end{enumerate}
\end{lemma}

\begin{proof}
\begin{enumerate}
\item follows from the construction of $\mb{F}$. 
\item Since $\mb{F}$ preserves cubes, we may restrict attention to a segment $[c;\beta]$ -- with $c\in X^n$ and $\beta =(\beta _1,\dots ,\beta^n): I\to I^n$ occuring in a presentation (cf Definition \ref{def:present}(3)) of $p$.  If $t<t'$ then $\beta_i(t)\le \beta_i(t')$.\\ Assume $\beta_i(t)\neq 0,1$. If $\beta_i(t')=1$, then $F(t,\beta_i(t))<1=F(t',\beta_i(t'))$ by Lemma \ref{lem:strict}(2-3). Otherwise, $F(t,\beta_i(t))\le F(t,\beta_i(t'))< F(t',\beta_i(t'))$ by Lemma \ref{lem:strict}(4-5). Hence $\vec{S}(p)$ has a presentation consisting of strict segments.
\item is a consequence of (2) for a path with a tame presentation (cf Definition \ref{def:present}(5)) since the map $\mb{F}$ preserves vertices.
\end{enumerate}
\end{proof}

\begin{lemma}
The map $\vec{S}:\vec{P}(X)_{x^-}^{x^+}\to\vec{S}(X)_{x^-}^{x^+}$ is continuous (in the compact open topologies) for every $\Box$-complex $X$ with source and target $x^-,x^+\in X_0$.
\end{lemma}

\begin{proof}
The map $\vec{S}$ corresponds by adjunction to the continuous map $I\times X^I\to X$ defined by $(t,p)\mapsto (t,p(t))\mapsto \mb{F}(t,p(t))$. It is continuous since $I$ is compact and Hausdorff. 
\end{proof}

\subsection{Strictification is a homotopy equivalence}
\begin{proposition}\label{prop:strict}
Let $X$ denote a pre-cubical set with vertices $x^-$ and $x^+$. Then the inclusions $\iota:\vec{S}(X)_{x^-}^{x^+}\hookrightarrow\vec{P}(X)_{x^-}^{x^+}$ and its restriction $\iota_T:\vec{ST}(X)_{x^-}^{x^+}\hookrightarrow\vec{T}(X)_{x^-}^{x^+}$ are homotopy equivalences.
\end{proposition}

\begin{proof} 
The homotopy $\mc{S}: I\times\vec{P}(X)_{x^-}^{x^+}\to\vec{P}(X)_{x^-}^{x^+}$ given by $\mc{S}(s,p)(t)=\mb{F}(st,p(t))$ connects the identity map (for $s=0$, apply Lemma \ref{lem:strict}(1)) with the map $\iota\circ S$ (for $s=1$). Its restriction to $\vec{S}(X)_{x^-}^{x^+}$ connects the identity map on that space with $S\circ\iota$. 

The restriction of $\mc{S}$ to a map from $\vec{T}(X)_{x^-}^{x^+}$ to $\vec{ST}(X)_{x^-}^{x^+}$ (well-defined because of Lemma \ref{lem:flow}(3)) is a homotopy inverse to the inclusion map $\vec{ST}(X)_{x^-}^{x^+}\hookrightarrow\vec{T}(X)_{x^-}^{x^+}$.
\end{proof}

\begin{remark}
\begin{enumerate}
\item Every d-path $p\in\vec{P}(X)_{x^-}^{x^+}$ can thus be arbitrarily well approximated by a strict d-path of the form $\mc{S}(s,p),\; s>0$. This shows that $\vec{S}(X)_{x^-}^{x^+}$ is  \emph{dense} in $\vec{P}(X)_{x^-}^{x^+}$. 
\item But $\vec{S}(X)_{x^-}^{x^+}$ is  \emph{not open} in $\vec{P}(X)_{x^-}^{x^+}$. Arbitrarily close to any strict d-path there is a d-path that ``pauses'' on a tiny interval.
 
\end{enumerate}
\end{remark}
\section{Cube chains and collars}
\subsection{The collar of a face in a cube}\label{ss:collar}
We propose a user-friendly notation for repeated face maps: first in a cube $I^n$ and then in a $\Box$-set $X$:
Every partition $[1:n]=J_0\sqcup J_*\sqcup J_1$ defines a face $d_{[J_0|J_*|J_1]}I^n=\{ 0\}^{J_0}\times I^{J_*}\times \{ 1\}^{J_1}$ of the cube $I^n$. Its (open) \emph{collar} $C_{[J_0|J_*|J_1]}I^n$ is defined as $[0,0.5[^{J_0}\times I^{J_*}\times ]0.5,1]^{J_1}\subset I^n$. In particular, the bottom vertex $\mb{0}$ is identified with $d_{[[1:n]|\emptyset |\emptyset ]}I^n$ and the top vertex with $\mb{1}=d_{[\emptyset |\emptyset |[1:n]]}I^n$. Remark that the only vertices in a collar $C_{[J_0|J_*|J_1]}I^n$ are those that are already present in the face $d_{[J_0|J_*|J_1]}I^n$.

For a $\Box$-set $X$, an $n$-cell $c$ in $X$ and a partition $J_0\sqcup J_*\sqcup J_1$, the combinatorics of the quotient map $q: I^n\to c$ gives rise to a face $d_{[J_0|J_*|J_1]}c=q(d_{[J_0|J_*|J_1]}I^n)$ with collar $C_{[J_0|J_*|J_1]}c=q(C_{[J_0|J_*|J_1]}I^n)$. 
Remark that, for a self-linked $\Box$-set, different partitions (of the same cardinality) can give rise to the same face.  

If $d$ and $c$ are cubes in $X$, the collar of $d$ in $c$ is defined as $C(d,c)=\bigcup_{[J_0|J_*|J_1]}C_{[J_0|J_*|J_1]}c$; the union is taken over all $[J_0|J_*|J_1]$ such that $d_{[J_0|J_*|J_1]}c$ is a face of $d$, including $d$ itself. The collar $C(d,X)=\bigcup_{c\in X_n, n\ge 0}C(d,c)$ of $d$ in $X$ agrees with a regular neighbourhood of $d$ with respect to a barycentric subdivision of the $\Box$-set $X$. The collar $C(x,X)$ of a vertex $x$ is called the \emph{star} $st(x)$ of $x$ in X. For simple illustrations, cf Figure \ref{fig:collar}.

\begin{center}
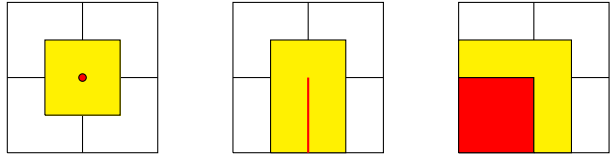
\begin{figure}[h]
\begin{tikzpicture}
\draw (0,0) -- (2,0) -- (2,2) -- (0,2) -- (0,0);
\draw (0,1) -- (2,1);
\draw (1,0) -- (1,2);
\draw[fill=yellow] (0.5,0.5) -- (1.5,0.5) -- (1.5,1.5) -- (0.5,1.5) -- (0.5,0.5);
\draw[fill=red] (1,1) circle [radius=0.05]; 
\draw (3,0) -- (5,0) -- (5,2) -- (3,2) -- (3,0);
\draw (3,1) -- (5,1);
\draw (4,0) -- (4,2);
\draw[fill=yellow] (3.5,0) -- (4.5,0) -- (4.5,1.5) -- (3.5,1.5) -- (3.5,0); 
\draw [thick, color=red] (4,0) -- (4,1);
\draw (6,0) -- (8,0) -- (8,2) -- (6,2) -- (6,0);
\draw (7,0) -- (7,2);
\draw (6,1) -- (8,1);
\draw [fill=red] (6,0) -- (7,0) -- (7,1) -- (6,1) -- (6,0);
\draw [fill= yellow] (7,0) -- (7.5,0) -- (7.5,1.5) -- (6,1.5) -- (6,1) -- (7,1) -- (7,0);
\end{tikzpicture}
\caption{Star of a vertex, collar of an edge and of a 2-cube in a Euclidean cubical complex (cf Example \ref{ex:cubset}(1)) consisting of four 2-cubes}\label{fig:collar}
\end{figure}
\end{center}

\begin{remark}\label{rem:open}
\begin{enumerate}
\item The collar $C(d,X)$ of a face $d$ in a $\Box$-set $X$ is open since it intersects every cube in $X$ in an open set; possibly empty.
\item If $c$ is a face of $c'$, then $C(c,X)\subseteq C(c',X)$.
\end{enumerate}
\end{remark}

\subsection{D-paths subordinate to a cube chain}
\begin{definition}\label{def:cc}  Let $X$ denote a $\Box$-set with two vertices $x^-, x^+\in X_0$ selected.
\begin{enumerate}
\item A \emph{cube chain}  $\mb{c}=(c_1,\dots , c_n)$ in $X$ from $x^-$ to $x^+$ \cite{Ziemianski:17} is a sequence of cubes\\ $c_i,\; 1\le i\le n,\; \dim{c_i}>0,$ with source and target vertices (cf Definition \ref{df:pcs}(1)) satisfying $c_{1,\mb{0}}=x^-, c_{n,\mb{1}}=x^+, c_{i-1,\mb{1}}=c_{i,\mb{0}}=x_i,\; 1< i\le n$.
\item A cube chain $\mb{c}=(c_1,\dots , c_n)$ in $X$ from $x^-$ to $x^+$ defines an associated \emph{vertex sequence} $(x^-=x_0,x_1,\dots ,x_{n-1},x_n=x^+)$ with $x_i=c_{\mb{i},1}=c_{\mb{i+1},0},\; 1\le i<n$.
\item The \emph{length} of a cube chain $\mb{c}$ is defined as $|\mb{c}|=\sum_{i=1}^n\dim c_i$.
\end{enumerate}
\end{definition}

\begin{remark}\label{rem:erase}
Only for a \emph{proper} $\Box$-set $X$ the correspondence from cube chains to vertex sequences is injective. 
\end{remark}

\begin{definition}
\begin{enumerate}
\item A \emph{track} $\mb{d}=(d_1,\dots , d_l)$ in $X$ from $x^-$ to $x^+$ \cite{Ziemianski:17} is a sequence of cubes $d_i,\; 0\le i\le l,$ with $d_{1\mb{0}}=x^-, d_{l\mb{1}}=x^+$ and such that some upper iterated face of $d^i$ agrees with some lower iterated face of $d^{i+1},\; 1\le i<l$; at most one of these two faces is allowed to be the original cube.
\item A track $\mb{t}_1^l$ in $X$ is called subordinate to the collar of the cube chain $\mb{c}_1^n$ in $X$ if there there exists a non-decreasing surjective map $j: [1:l]\to [1:n]$ such that, for $1\le i\le l$,  $t_i$ is a coface of $c_{j(i)}$ or of one of its iterated faces.
\item A path $p\in C(\mb{c},X)$ is called subordinate to the collar of the cube chain $\mb{c}$ in $X$ if it allows a presentation $p=[t_1;\beta_1]*_{s_1}[t_2;\beta_2]*_{s_2}\dots *_{s_{l-1}}[t_l;\beta_l]$ with $\mb{t}=(t_i)_1^l$ a track subordinate to $\mb{c}$, $d_{[J_0,J_*,J_1]}t_i=c_{j_(i)}$ and $\beta_i\in\vec{P}_{[s_{i-1},s_i]}( C(d_{[J_0,J_*,J_1]}I^{\dim t_{j_(i)}}, I^{\dim t_{j_(i)}})),$\\ $\beta_i(s_{i-1})\in st(d_{[J_0\cup J_*,\emptyset ,J_1]}I^{\dim t_{j_(i)}}), \beta_i(s_i)\in st(d_{[J_0,\emptyset ,J_*\cup J_1]}I^{\dim t_{j_(i)}})$. 
\end{enumerate}
\end{definition}



\begin{center}
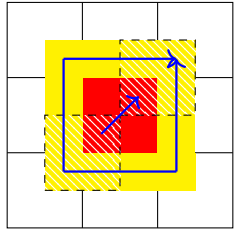
\begin{figure}[h]\label{fig:subocc}
\begin{tikzpicture}
\draw (0,0) rectangle (3,3);
\draw (1,0) -- (1,3);
\draw (2,0) -- (2,3);
\draw (0,1) -- (3,1);
\draw (0,2) -- (3,2);
\fill[red] (1,1) rectangle (2,2);
\fill[yellow] (0.5,0.5) rectangle (2.5,1);
\fill[yellow] (0.5,2) rectangle (2.5,2.5);
\fill[yellow] (0.5,0.5) rectangle (1,2.5);
\fill[yellow] (2,0.5) rectangle (2.5,2.5);
\draw [pattern=north west lines, pattern color = white, dashed] (0.5,0.5) rectangle (1.5,1.5);
\draw [pattern=north west lines, pattern color = white, dashed] (1.5,1.5) rectangle (2.5,2.5);
\draw[->, line width=0.3mm, color=blue] (1.25,1.25) -- (1.75,1.75);
\draw[line width=0.3mm, color=blue] (0.75,0.75) -- (0.75,2.25);
\draw[->, line width=0.3mm, color=blue] (0.75,2.25) -- (2.25,2.25);
\draw[line width=0.3mm, color=blue] (0.75,0.75) -- (2.25,0.75);
\draw[->, line width=0.3mm, color=blue] (2.25,0.75) -- (2.25,2.25);
\end{tikzpicture}
\caption{Cube $c$ in red, collar $C(c)$ in yellow, stars $st(c_{\mb{0}})$ and $st(c_{\mb{1}})$ dashed. The two outer d-paths (in blue) are subordinate to the collars of the cube chains $(d_{[2,1,\emptyset ]}c,d_{[\emptyset , 2,1]}c)$, resp.\ $(d_{[1,2,\emptyset ]}c, _{[\emptyset , 1,2]}c)$. All three d-paths are subordinate to the cube chain $(c)$ consisting of a single cube.}
\end{figure}
\end{center}

\begin{remark}
Ziemia\'{n}ski shows in \cite[Proposition 3.5]{Ziemianski:19} that every non-constant d-path is contained in a track.
\end{remark} 

\begin{definition}\label{def:subo}
Let $\mb{c}=(c_1,\dots ,c_n)$ denote a cube chain in a $\Box$-set $X$ with associated vertex sequence $(x^-=x_0,x_1,\dots ,x_n=x^+)$. 
\begin{enumerate}
\item The subspace $\vec{T}_{\mb{c}}(X)\subset\vec{T}(X)$ of (tame) d-paths subordinate to $\mb{c}$ consists of d-paths with presentation $p=[c_1;p_1]*_{t_1}\dots *_{t_{n-1}}[c_n;p_n]$ with $p_i\in\vec{P}_{[t_{i-1},t_i]}(I^{\dim c_i})_{\mb{0}}^{\mb{1}}$.
\item The subspace $\vec{P}_{\mb{c}}^C(X)\subset\vec{P}(X)$ of d-paths subordinate to the \emph{collar} $C(\mb{c})$ of $\mb{c}$ consists of d-paths $p=p_1*_{t_1}\dots *_{t_{n-1}}p_n$ with $p_i\in\vec{P}_{[t_{i-1},t_i]}(C(c_i,X))_{y_{i-1}}^{y_i}$ with  $y_i\in st(x_i)$. 
\item $\vec{ST}_{\mb{c}}(X)_{x^-}^{x^+}:=\vec{T}_{\mb{c}}(X)\cap\vec{S}(X)$, $\vec{S}_{\mb{c}}^C(X)_{x^-}^{x^+}:=\vec{P}_{\mb{c}}^C(X)\cap\vec{S}(X)_{x^-}^{x^+}$,\\ $\vec{T}_{\mb{c}}^C(X)_{x^-}^{x^+}:=\vec{P}_{\mb{c}}^C(X)\cap\vec{T}(X)_{x^-}^{x^+},$ $\vec{ST}_{\mb{c}}^C(X)_{x^-}^{x^+}:=\vec{P}_{\mb{c}}^C(X)\cap\vec{ST}(X)_{x^-}^{x^+}$.
\end{enumerate}
\end{definition} 

\begin{figure}[h]
\begin{tikzpicture}
\draw (0,0) -- (2,0) -- (2,4) -- (0,4) -- (0,0);
\draw (0,2) -- (2,2);
\draw (3,0) -- (5,0) -- (5,4) -- (3,4) -- (3,0);
\draw (3,2) -- (5,2);
\draw (6,0) -- (8,0) -- (8,4) -- (6,4) -- (6,0);
\draw (6,2) -- (8,2);
\draw (9,0) -- (11,0) -- (11,4) -- (9,4) -- (9,0);
\draw (9,2) -- (11,2);
\draw[->, line width=0.6mm, color=blue] (2,2) -- (2,4);
\draw[->, line width=0.6mm, color=blue] (0,0) -- (2,2);
\draw[->, line width=0.6mm, color=blue] (3,0) -- (3,2);
\draw[->, line width=0.6mm, color=blue] (3,2) -- (5,2);
\draw[->, line width=0.6mm, color=blue] (5,2) -- (5,4);
\draw[->, line width=0.6mm, color=blue] (6,0) -- (8,4);
\draw[->, line width=0.6mm, color=blue] (9,0) -- (11,4);
\draw[->, line width=0.6mm, color=magenta] (9,0) -- (9.5,1.5);
\draw[->, line width=0.6mm, color=magenta] (9.5,1.5) -- (11,2);
\draw[->, line width=0.6mm, color=magenta]  (11,2) -- (11,4);
\draw[fill=gray] (6,3) -- (7,3) -- (7,4) -- (6,4) -- (6,3);
\draw[fill=gray] (9,3) -- (10,3) -- (10,4) -- (9,4) -- (9,3);
\draw[fill=gray] (10,0) -- (11,0) -- (11,1) -- (10,1) -- (10,0);
\end{tikzpicture}
\caption{d-paths subordinate to the cube chain $\mb{c}$ consisting of a 2-cube and a 1-cube, resp.\ to the cube chain $\mb{d}$ consisting of three 1-cubes; moreover subordinate to their respective collars. The path in magenta is contained in $\vec{T}_{\mb{d}}^C(X)$, but not in $\vec{T}_{\mb{d}}(X)$.}
\end{figure}
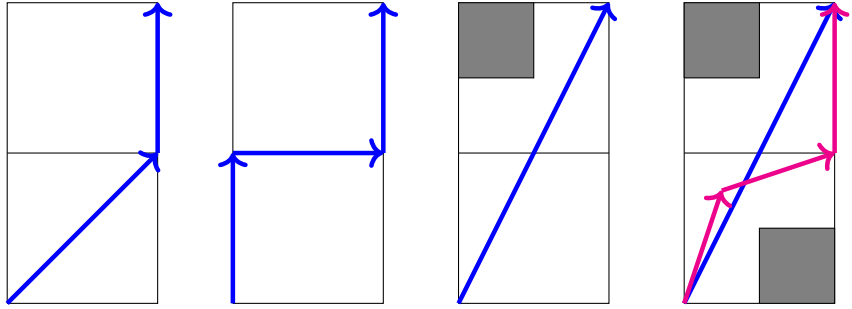

\begin{remark}
\item For every d-path subordinate to the collar of $\mb{c}$, there exists a covering of the interval $I$ by open intervals $K_i$ such that $p_i(K_i)\subset C(c_i,X)$ and $p_i(K_i\cap K_{i+1})\subset st(x_{i+1})$.
\end{remark}

\begin{example}
A cube chain may include a loop. This is the major reason why we investigate paths \emph{subordinate to} the collar of a cube chain and not just paths \emph{contained} in the collar of a cube chain. Let us have a closer look at a particular case:

Let $Z_2$ denote the 2-dimensional $\Box$-set with one $0$-cell $v$, one $1$-cell $e$ and one 2-cell $c$ with quotient map $q: I^2\to Z_2$ (cf Example \ref{ex:cubset}(3)). The star of $v$ has the form\\ $st(v)=q(I^2\setminus \{(x,y)|\; x\neq 0.5, y\neq 0.5\} )$. The edge $e$ has collar $C(e,X)=q(I^2\setminus \{ (0.5,0.5)\})$. A d-path $p\in \vec{P}(Z_2)_v^v$ in $X$ subordinate to the collar of the cube chain $(e)$ (consisting of a single 1-cube) can only reach $v$ from $v$ running once through $e$. It is subordinate to the collar of the cube chain $(e,e)$ if $p=q\circ\alpha$ with $\alpha\in\vec{P}(I^2\setminus \{ (0.5,0.5)\} )_{\mb{0}}^{\mb{1}}$. 

A d-path subordinate to the collar of $(e)$ and another one subordinate to the collar of $(e,e)$ cannot be d-homotopic since they have different lengths \cite[Proposition 2.12]{Raussen:09}.
\end{example}

\begin{proposition}\label{prop:open} Let $\mb{c}$ denote a cube chain in $X$ from $x^-$ to $x^+$. 
\begin{enumerate}
\item The path space $\vec{P}_{\mb{c}}^C(X)_{x^-}^{x^+}\subset\vec{P}(X)_{x^-}^{x^+}$ subordinate to the collar of $\mb{c}$ is open.
\item Likewise are $\vec{S}_{\mb{c}}^C(X)_{x^-}^{x^+}\subset\vec{S}(X)_{x^-}^{x^+}$, $\vec{T}_{\mb{c}}^C(X)_{x^-}^{x^+}\subset\vec{T}(X)_{x^-}^{x^+}$, and $\vec{ST}_{\mb{c}}^C(X)_{x^-}^{x^+}\subset\vec{ST}(X)_{x^-}^{x^+}$.
\end{enumerate}
\end{proposition}

\begin{proof}
\begin{enumerate}
\item d-paths in $\vec{P}_{\mb{c}}^C(X)_{x^-}^{x^+}$ are characterized by intersecting cubes in tracks along $\mb{c}$ in open subsets; they are concatenated along open stars of vertices. 
\item by definition of the topology induced on subspaces. 
\end{enumerate} 
\end{proof}

\begin{proposition}\label{prop:vs} 
\begin{enumerate}
\item Every cube chain $\mb{c}$ -- and hence also its collar \(C(\mb{c})\) -- contains a tame strict d-path.
\item For every strict d-path $p\in\vec{S}(X)_{x^-}^{x^+}$, there exists a cube chain $\mb{c}(p)$ such that $p\in\vec{S}_{\mb{c}(p)}^C(X)_{x^-}^{x^+}$.
\end{enumerate}
\end{proposition}

In the following proof, we will make use of \emph{coordinate hyperplanes} in a cube chain. In a cube $I^n$, consider middle hyperplanes given by the equations $x_i=0.5,\; 1\le i\le n$. The elementary but crucial observation is that a \emph{strict d-path} intersects any coordinate hyperplane in \emph{at most} one point. Likewise, one defines coordinate middle hyperplanes (potentially with identifications on the boundary, one boundary middle hyperplane being identified with another such) in each cube in a $\Box$-set $X$. 

\begin{proof}
\begin{enumerate}
\item The diagonal path $\delta_n$ in $\vec{ST}(I^n)_{\mb{0}}^{\mb{1}}$ connects bottom and top vertex of the $n$-cube diagonally with constant speed. For an $n$-cell $c$ in $X$, composition with the quotient map $I^n\downarrow c$ defines a strict tame path $\delta (c)$ in $c$ from its bottom vertex to its top vertex. The concatenation $\delta (\mb{c}):=\delta (c_0)*\dots *\delta (c_n)$ defines  a strict tame path in the cube chain $\mb{c}$ connecting bottom and top vertex. 
\item For every \emph{strict} d-path $p=(p_1,\dots ,p_n): J=[j^-,j^+]\to I^n$ in a single cube $I^n$ defined on some interval $J\subseteq I$, there is a finite (ordered) set $S\subset J$ (possibly empty) consisting of $s_j\in J$ at which $p$ intersects one or several of the middle hyperplanes $x_i=0.5,\; 1\le i\le n.$
Define $J^j_0, J^j_*, J^j_1$ as the set of indices $i$ for which $p_i(s_j)$ is less than, equal, resp.\ greater than $0.5$. Since $p$ is directed, we have that $J_1^{j+1}=J_1^j\cup J_*^j$ and $J_0^{j+1}=J_0^j\setminus J_*^{j+1}$. For $\max (j^-,s_{j-1})<t<\min (j^+,s_{j+1})$,  $p(t)$ is contained in the collar of the face $d_{[J^j_0|J^j_*|J^j_1]}I^n$  -- the minimal face with this property. 

For $\max (j^-,s_{j-1})<t<\min (j^+,s_j)$, $p(t)$ is contained in the star of the \emph{vertex} $d_{[J^j_0\cup J^j_*|\emptyset |J^j_1]}I^n$, and for $\max (j^-,s_j)<t<\min (j^+,s_{j+1})$ in the star of the \emph{vertex} $d_{[J^j_0|\emptyset |J^j_*\cup J^j_1]}I^n$. The entire path is therefore contained in the collar $\mb{c}(p)$ of the cube chain defined by the subsequent cubes $d_{[J^j_0|J^j_*|J^j_1]}I^n$. 

Two special cases deserve particular attention:
\begin{enumerate}
\item This cube chain degenerates to a single vertex $d_{[J_0|\emptyset |J_1]}I^n$ if $p$ does not intersect any of the hyperplanes $x_i=0.5$.
\item If $s_{min}=j^-$ and $p(j^-)$ is contained in a lower boundary face,  then the first cube $d_{[J^{min}_0|J^{min}_*|J^{min}_1]}I^n$ of the cube chain $\mb{c}(p)$ is the minimal face containing $p(j^-)$. If $s_{max}=j^+$ and $p(j^+)$ is contained in an upper boundary face, then the last cube $d_{[J^{max}_0|J^{max}_*|J^{max}_1]}I^n$ of $\mb{c}(p)$ is minimal containing $p(j^+)$.
\end{enumerate}

Now let $p:I\to X$ denote a strict d-path in a $\Box$-set $X$ from $x^-$ to $x^+$ with presentation $p=[c_1;p^1]*_{t_1}[c_2,p^2]*_{t_2}\dots *_{t_{l-1}}[c_l,p^l]$.  Then the construction above can be performed for each individual cube $c_i$ leading to a sequence $\mb{c}(p)=\mb{c}(p^1)*\dots *\mb{c}(p^l)$ of cubes the collar of which contains $p(I)$. One has to check that two subsequent cubes ``match'':

If $p^i(t_i)=p^{i+1}(t_i)$ is not contained in any middle hyperplane, then it is contained in the star of a vertex which is the top vertex in the cube chain corresponding to $p^i$ and the bottom vertex in that corresponding to $p^{i+1}$. If $p^i(t_i)=p^{i+1}(t_i)$ is contained in a middle hyperplane, then the last cube in the cube chain corresponding to $[c_i,p^i]$ agrees with the first one in the cube chain corresponding to $[c_{i+1},p_{i+1}]$, ie the minimal cube in the boundary of $c_i$ and of $c_{i+1}$ containing $p_i(t_i)$ -- according to (b) above.

In a final step, one erases cubes consisting of a single vertex; cf (a) above.
\end{enumerate}
\end{proof}

\begin{center}
\begin{figure}[h]
\begin{tikzpicture}
\fill[yellow] (-5,0) rectangle (-1,1);
\draw (-5,0) rectangle (-1,2);
\draw (-3,0) -- (-3,2);
\draw[line width = 0.6mm, color=orange] (-5,0) -- (-1,0);
\draw[->, line width =0.6mm, color =blue] (-5,0.2) -- (-1,0.8);
\fill[yellow] (0,0) -- (4,0) -- (4,2) -- (1,2) -- (1,1) -- (0,1) -- (0,0);
\fill[yellow] (5,0) -- (8,0) -- (8,1) -- (9,1) -- (9,2) -- (6,2) -- (6,1) -- (5,1) -- (5,0);
\draw (0,0) rectangle (4,2);
\draw[line width = 0.6mm, color=orange] (2,0) rectangle (4,2);
\draw[line width = 0.6mm, color=orange] (0,0) -- (2,0);
\draw[->, line width =0.6mm, color=blue] (2,0.8) -- (4,1.8);
\draw[->, line width =0.6mm, color=blue] (0,0.2) -- (2,0.8);
\draw (5,0) rectangle (9,2);
\draw[->, line width =0.6mm, color =blue] (5,0.5) -- (9,1.5);
\draw[line width = 0.6mm, color=orange] (5,0) -- (7,0) -- (7,2) -- (9,2);
\draw (6.97,0) -- (6.97,2);
\draw (7.03,0) -- (7.03,2);
\draw[fill=red] (-3,0) circle [radius=0.05];
\draw[fill=red] (2,0) circle [radius=0.05];
\draw[fill=red] (7,0) circle [radius=0.05];
\draw[fill=red] (7,2) circle [radius=0.05];
\end{tikzpicture}
\caption{D-paths in two subsequent cubes in blue; the corresponding cube chains in orange, their collars in yellow. In the first two cases the two cube chains associated to each individual square consist of a single cube with common top, resp.\ bottom vertex; in the last case, the two cube chains share a common edge cube.}
\end{figure}
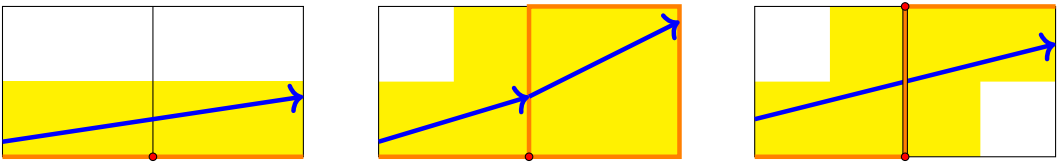
\end{center}

\begin{remark}
The analogue of Proposition \ref{prop:vs}(2) is wrong for non-strict d-paths. Let $X$ be the cubical set (consisting of two 2-cubes) corresponding to $[0,2]\times [0,1]$. The d-path in Figure \ref{fig:non} that linearly connects $(0,0), (0,0.5), (2,0,5)$ and $(2,1)$ is not contained in $\vec{P}^C_{\mb{c}}(X)_{(0,0)}^{(2,1)}$ for any cube chain connecting $(0,0)$ with $(2,1)$.
\end{remark}

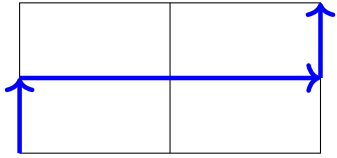
\begin{figure}[h]
\begin{tikzpicture}
\draw (0,0) -- (2,0) -- (4,0) -- (4,2) -- (2,2) -- (0,2) -- (0,0);
\draw (2,0) -- (2,2);
\draw[->, line width=0.6mm, color=blue] (4,1) -- (4,2);
\draw[->, line width=0.6mm, color=blue] (0,0) -- (0,1);
\draw[->, line width=0.6mm, color=blue] (0,1) -- (4,1);
\end{tikzpicture}
\caption{A d-path that is not contained in the collar of any cube chain from bottom to top}\label{fig:non}
\end{figure}

\subsection{A poset category of cube chains}\label{ss:pathcol}
\begin{definition}\label{def:pcc}
\begin{enumerate}
\item An \emph{elementary refinement} of a cube $c$ consists of two subsequent faces $d_{[J_0|[1:n]\setminus J_0|\emptyset ]}c$ and $d_{[\emptyset |J_0|[1:n]\setminus J_0]}c$ with $\emptyset\neq J_0\neq [1:n]$.  
\item An elementary refinement \emph{of a cube chain} $\mb{c}$ arises by replacing a single cube $c_i$ by one of its elementary refinements. 
\item A \emph{refinement} of a cube chain arises by reflexive and transitive closure of elementary refinements. \item Refinement between cube chains in $X$ from a vertex $x^-$ to a vertex $x^+$ defines a \emph{partial order} relation among cube chains that gives rise to a thin (poset) category $\mc{C}(X)_{x^-}^{x^+}$ with cube chains as objects and refinements as morphisms. 
\end{enumerate}
\end{definition}
Remark that, for a general $\Box$-set $X$, the category $Ch(X)$ in Ziemia\'{n}ski's \cite[Section 7]{Ziemianski:19} differs from  
this poset category. Ziemianski's cube chains are concatenations of \emph{cubical maps} from a standard cube into a cube in $X$; moreover cubical symmetries of the standard cube that induce identities in the $\Box$-set are an additional part of the structure.  
\begin{proposition}\label{prop:refine}
All cube chains are supposed to be cube chains in $X$ from $x^-$ to $x^+$.
\begin{enumerate}
\item $\mb{c}'$ refines $\mb{c}$ if and only if $\vec{S}_{\mb{c}'}^C(X)_{x^-}^{x^+}\subseteq\vec{S}_{\mb{c}}^C(X)_{x^-}^{x^+}.$
\item If $\mb{c}'$ is a proper refinement of $\mb{c}$, then $\vec{S}_{\mb{c}'}^C(X)_{x^-}^{x^+}\subset\vec{S}_{\mb{c}}^C(X)_{x^-}^{x^+}$ is a proper subset.
\item For every strict d-path $p\in\vec{S}(X)$, the chain $\mb{c}(p)$ (Proposition \ref{prop:vs}(2)) is \emph{finest} among the cube chains $\mb{c}$ with $p\in\vec{S}_{\mb{c}}^C(X)_{x^-}^{x^+}$.
\item Path spaces $\vec{S}_{\mb{c}}^C(X)_{x^-}^{x^+}, \vec{S}_{\mb{c}'}^C(X)_{x^-}^{x^+}$ intersect iff $\mb{c}$ and $\mb{c}'$ have a common refinement. Likewise for spaces of strict tame paths.
\item $\vec{S}_{\mb{c}''}^C(X)_{x^-}^{x^+}\subseteq\vec{S}_{\mb{c}}^C(X)_{x^-}^{x^+}\cap\vec{S}_{\mb{c}'}^C(X)_{x^-}^{x^+}$ for a common refinement $\mb{c}''$ of $\mb{c}$ and $\mb{c}'$. If $\mb{c''}$ is a coarsest common refinement (if a such exists), then $\vec{S}_{\mb{c}''}^C(X)_{x^-}^{x^+}=\vec{S}_{\mb{c}}^C(X)_{x^-}^{x^+}\cap\vec{S}_{\mb{c}'}^C(X)_{x^-}^{x^+}$.
\item If $X$ is proper and non-self-linked (cf Definition \ref{df:pcs}(2-3)), then two cube chains either do not have a common refinement or they have a coarsest one (for which equality holds in (5)).
\item Similar results hold for spaces of strict and tame d-paths.
\end{enumerate}
\end{proposition}

\begin{proof}
\begin{enumerate}
\item $\Rightarrow$ from the definition of collars and paths in collars in Definition \ref{def:subo}.\\
$\Leftarrow$: Assume $\mb{c}'=(c_1',\dots , c_{l'})$ does not refine $\mb{c}=(c_1,\dots ,c_l)$ and such that the prefix $(c_1',\dots ,c'_{k'})$ refines a minimal prefix of $(c_1,\dots ,c_k)$ but $c'_{k'+1}$ is not a face of  $c_k$ nor of $c_{k+1}$. Then the diagonal path $\delta (c'_1)*\dots *\delta (c'_{l'})\in \vec{S}_{\mb{c}'}^C(X)_{x^-}^{x^+}$ is not contained in $\vec{S}_{\mb{c}}^C(X)_{x^-}^{x^+}$.  
\item The diagonal path in $\vec{ST}_{\mb{c}}(X)_{x^-}^{x^+}\subseteq \vec{S}_{\mb{c}}^C(X)_{x^-}^{x^+}$ from the proof of Proposition \ref{prop:vs}(1) is not contained in $\vec{S}_{\mb{c}'}^C(X)_{x^-}^{x^+}$.
\item Every elementary refinement (cf above) of the cube chain $\mb{c}(p)$ described above results in a collar that does not intersect at least one of the hyperplanes $x_j=0.5$ with $j\in J_*$ within a cube $c_i$. In particular, it cannot contain $p(I)$ which has a non-empty intersection with  this hyperplane.
\item $\Leftarrow$ follows from (1) and Proposition \ref{prop:vs}(1).\\ $\Rightarrow$: By (3) above, the cube chain $\mb{c}(p)$ for a path $p$ in the intersection refines both $\mb{c}$ and $\mb{c}'$.
\item The first part follows from (1) above. Let  $p\in\vec{S}_{\mb{c}}^C(X)_{x^-}^{x^+}\cap\vec{S}_{\mb{c}'}^C(X)_{x^-}^{x^+}$. By Proposition \ref{prop:vs}(2), $p\in\vec{S}_{\mb{c}(p)}^C(X)_{x^-}^{x^+}$, and by (3) above, $\mb{c}(p)$ refines both $\mb{c}$ and $\mb{c}'$ and thus $\mb{c}''$. Apply (1) above to $\mb{c}(p)$ and $\mb{c}''$.
\item First note that the length of a cube chain (cf Definition \ref{def:cc}(4)) is invariant under refinements. The proof proceeds by induction on this length. The statement is obviously true for cube chains of length $0$ and $1$. Let $\mb{c}=(c_1,\dots ,c_n), \mb{c'}=(c'_1,\dots ,c'_n)$ denote two cube chains with a common refinement. There is a number of vertices (at least one) such that the edge from $x^-$ to that vertex refines the first face of some common refinement of $c_1$ and of $c_1'$. All of these vertices span a (unique since $X$ is proper) maximal common lower cube $d$ of $c_1$ and $c_1'$ giving rise to elementary refinements $(d,d_1)$ of $c_1$ and $(d,d_1')$ of $c_1'$. Then the cube chains arising by replacing $c_1$ by $d_1$ and $c_1'$ by $d_1'$ have a shorter length and hence a coarsest common refinement. Add $d$ to the resulting cube chain at the beginning.
\end{enumerate}
\end{proof}

\begin{example}
Let $X$ denote the non-proper $\Box$-set from Example \ref{ex:cubset}(2) consisting of two 2-cubes $c^1, c^2$ glued along a common boundary consisting of four edges. Then the cube chains consisting solely of $c^1$, resp.\ of $c^2$ possess two common refinements (consisting of two consecutive boundary edges) but no coarsest such.
\end{example}

\subsection{Path spaces as colimits}
We define functors into the category $\mathbf{Top}$ of topological spaces: $\vec{S}, \vec{T}$ and $\vec{ST}: \mc{C}(X)_{x^-}^{x^+}\to\mathbf{Top}$ by\\
\(\vec{S}(\mb{c})=\vec{S}_{\mb{c}}^C(X)_{x^-}^{x^+}\), $\vec{T}(\mb{c})=\vec{T}_{\mb{c}}^C(X)_{x^-}^{x^+}$, and \(\vec{ST}(\mb{c})=\vec{ST}_{\mb{c}}^C(X)_{x^-}^{x^+}\).\\
Refinement of cube sequences is reflected in inclusion of path spaces (Proposition \ref{prop:refine}(1)).\\ As a consequence of Proposition \ref{prop:vs} and \ref{prop:refine}, we conclude: 

\begin{corollary}
Let $X$ denote a $\Box$-set with vertices $x^-, x^+\in X_0$. Then 
\[\vec{S}(X)_{x^-}^{x^+}=\colim_{\mc{C}(X)_{x^-}^{x^+}}\vec{S},\qquad  \vec{T}(X)_{x^-}^{x^+}=\colim_{\mc{C}(X)_{x^-}^{x^+}}\vec{T},\qquad \text{ and }\vec{ST}(X)_{x^-}^{x^+}=\colim_{\mc{C}(X)_{x^-}^{x^+}}\vec{ST}.\] 
\end{corollary}
The colimit identifies path spaces in finer cube chains with subspaces of path spaces in coarser ones; the colimit is therefore just a union of topological spaces.

\section{Comparing paths in a cube chain and in its collar}
Let $\mb{c}=(c_1,\dots ,c_n)$ denote a cube chain in a $\Box$-set $X$ between vertices $x^-$ and $x^+$. The purpose of this section is to compare spaces of strict paths \emph{within} a cube chain $\mb{c}$ with those \emph{subordinate} to it. The notation was fixed in Definition \ref{def:subo}.

\begin{proposition}\label{prop:contr}
\begin{enumerate}
\item $\vec{ST}_{\mb{c}}(X)_{x^-}^{x^+}\hookrightarrow\vec{S}_{\mb{c}}^C(X)_{x^-}^{x^+}$ is a deformation retract.
\item $\vec{ST}_{\mb{c}}(X)_{x^-}^{x^+}\hookrightarrow\vec{ST}_{\mb{c}}^C(X)_{x^-}^{x^+}$ is a deformation retract.
\end{enumerate}
\end{proposition}

\begin{proof}
(1) will be proved through a cubewise taming construction - first for individual d-paths and then for spaces of such. The proof of (2) follows the same pattern; since taming is performed cubewise, the d-paths and d-homotopies in the construction below taking departure in tame d-paths \emph{stay tame}.

To prove (1), we consider in a first step only strict d-paths subordinate to the collar of a cube chain $\mb{d}=(d_1,\dots ,d_k)$ \emph{within the standard cube} $I^n$; $d_j$ is the face $d_{[J_0^j|J_*^j|J_1^j]}I^n$ with bottom vertex $v_{j-1}=d_{[J_0^j\cup J_*^j|\emptyset |J_1^j]}I^n$ and top vertex $v_j=d_{[J_0^j|\emptyset |J_*^j\cup J_1^j]}I^n$. 

For each cube $d_j$, we define a taming map (its geometric interpretation is explained in Remark \ref{rem:taming}(1)) $\tau_j:\vec{S}_{[a_j,b_j]}(I^n)_{st(v_{j-1})}^{st(v_j)}\to \vec{ST}_{[a_j,b_j]}(d_j)_{v_{j-1}}^{v_j}$ by associating to the path $p=(p_i)$ the path $\tau (p)=q=(q_i)$ given by 
\begin{equation}\label{eq:t(p)} q_i(t)=\begin{cases}0 & i\in J_0\\ \frac{p_i(t)-p_i(a_j)}{p_i(b_j)-p_i(a_j)} & i\in J_*\\ 1 & i\in J_1 \end{cases}.\end{equation}

Where to fit these constructions for consecutive cubes $d_j$, which domain intervals $[a_j,b_j]$ should one choose? If $k=1$ (only one cube in the chain), it is the entire domain interval. 

For $k>1$, consider the sequence of piecewise linear hypersurfaces $H_j\subset I^n,\; 1\le j< k,$ given by the equations $m_j(x)=1, x\in I^n,$ with $m_j(x_1,\dots ,x_n):=\min_{i\in J^j_*}x_i+\max_{i\in J^{j+1}_*}x_i$.
Remark that the collars of $d_j$ and of $d_{j+1}$ intersect  in the star $st(v_j)$ of $v_j$.

For a strict d-path, $p\in \vec{S}^C_{\mb{d}}(I^n)$, consider the functions given by the compositions $m_j\circ p: I\to\mb{R},\; 1\le j\le k$; they are \emph{strictly increasing}. When $p$ enters $\overline{st(v_j)}$ at $t=t_j^-$, we have $\min_{i\in J^j_*}p_i(t_j^-)=0.5$, whereas $ \max_{i\in J^{j+1}_*}p_i(t_j^-)<0.5$; their sum being less than $1$. When $p$ exits $\overline{st(v_j)}$ at $t=t_j^+$, we have $\max_{i\in J^{j+1}_*}p_i(t_+)=0.5$ and $\min_{i\in J^j_*}p_i(t_j^+)>0.5$; their sum being greater than $1$. We conclude that there exists a unique ascending sequence $t_1<t_2<\dots <t_{k-1}$ such that $m_j(p(t_j))=\min_{i\in J^j_*}p_i(t_j)+\max_{i\in J^{j+1}_*}p_i(t_j)=1,\; 1\le j<k$. Remark that $p(t_j)\in st(v_j)$. 

\begin{center}
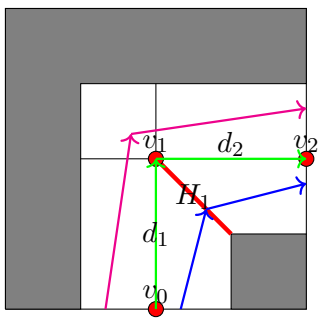
\begin{figure}[h]
\begin{tikzpicture}
\draw (0,0) -- (4,0) --  (4,4) -- (0,4) -- (0,0);
\draw (2,0) -- (2,4);
\draw (0,2) -- (4,2);
\draw[fill=gray] (0,0) -- (1,0) -- (1,3) -- (4,3) -- (4,4) -- (0,4) -- (0,0);
\draw[fill=gray] (3,0) -- (4,0) -- (4,1) -- (3,1) -- (3,0);
\draw[line width=0.6mm, color=red] (2,2) -- (3,1);
\node at (2.5,1.5) {$H_1$};
\draw[fill=red] (2,0) circle [radius=0.1]; 
\draw[fill=red] (2,2) circle [radius=0.1]; 
\draw[fill=red] (4,2) circle [radius=0.1]; 
\draw[->, magenta, line width=0.3mm]  (1.67,2.33) -- (4,2.67);
\draw[->, magenta, line width=0.3mm]  (1.33,0) -- (1.67,2.33); 
\draw[->, blue, line width=0.3mm] (2.67,1.33) -- (4,1.67);
\draw[->, blue, line width=0.3mm] (2.33,0) -- (2.67,1.33); 
\draw[->, green, line width=0.3mm] (2,2) -- (4,2);
\draw[->, green, line width=0.3mm] (2,0) -- (2,2);
\node at (2,0.2) {$v_0$};
\node at (2,2.2) {$v_1$};
\node at (4,2.2) {$v_2$};
\node at (2,1){$d_1$};
\node at ((3,2.2){$d_2$};
\end{tikzpicture}
\caption{Cube chain $(d_1,d_2)$, its collar, line hypersurface $H_1$ within a 2-cube, d-paths (blue, within one cube; magenta through three cubes) with the same taming (green)}
\end{figure}
\end{center}

We define then the \emph{tame} $d$-path $\tau (p)=q$ as the concatenation $q_1*_{t_1}q_2*_{t_2}\dots *_{t_{k-1}}q_k$ of strict d-paths $q_j=\tau (p|[t_{j-1},t_j])$ (from $v_{j-1}$ to $v_j$). Furthermore, we define a linear d-homotopy $H=H_s$ of strict d-paths $H_s\in\vec{S}_{\mb{d},[t_0,t_k]}^C(I^n)$ connecting $p$ and $q=\tau (p)$ as\\ $H(t,s)=(1-s)p(t)+sq(t)$.

\end{proof}

\begin{remark}\label{rem:taming}
\begin{enumerate}
\item Formula (\ref{eq:t(p)}) has the following interpretation: On the interval $[t_j,t_{j+1}]$, the components of 
$p$ in $I^j_0$ resp.\ in $I^j_1$ are compressed to $0$, resp.\ $1$ (the respective processes in an HDA are idle), whereas its components in $I^j_*$ are stretched to fill the entire interval $[0,1]$ (ie, the respective processes take a full step). Remark that $q(t_j)$ is a vertex for every $0\le j\le k$. 
\item The definition of the taming function $\tau$ and of the taming homotopy $H$ above on the collar of a cube chain $\mb{d}'$ that is finer than $\mb{d}$ is the restriction of the respective functions and homotopies corresponding to $\mb{d}$. 
\item For a strict d-path $p$ that is already contained in the cube chain $\mb{c}$ itself, $p(t)$ solves the equation $\min_{i\in J^j_*}x_i+\max_{i\in J^{j+1}_*}x_i=1$ exactly at $t_j$ with $p(t_j)=v_j$. Hence $\tau (p)=p$ and $H(t,s)=p(t)$ for $s\in I$.
\end{enumerate}
\end{remark}

\begin{lemma}\label{lem:cont}
For every cube chain $\mb{d}$ in $I^n$, the times $t_j=t_j(p),\; 1\le j<k,$ define \emph{continuous} functions $t_j: \vec{S}_{\mb{d}}^C(I^n)_{st(v_0)}^{st(v_k)}\to I,\; p\mapsto t_j(p)$. 
\end{lemma}

\begin{proof}
Given a d-path $p\in\vec{S}^C_{\mb{d}}(I^n)$ and $\varepsilon>0$ consider the \emph{open} set of all strict d-paths $q$ satisfying $m_j(q(t_j(p)-\varepsilon ))<1$ and $m_j(q(t_j(p)+\varepsilon ))>1$. Obviously, it contains the path $p$. For a strict d-path satisfying these two inequalities, the solution $t=t_j(q)$ of $m_j(q(t))=1$ is contained in the interval $(t_j(p)-\varepsilon ,t_j(p)+\varepsilon )$.
\end {proof}

\begin{proof}[Proof of Proposition \ref{prop:contr} continued] For a given strict d-path $p\in\vec{S}_{\mb{c}}^C(X)_{x^-}^{x^+}$, fix a presentation (cf Definition \ref{def:present}(3)) and perform the taming construction above cubewise. Because of Remark \ref{rem:taming}(2), the resulting tamed path and d-homotopy does not depend on the chosen presentation. Moreover, the tamed d-paths in subsequent cubes (and the d-homotopies) ``fit'' at top, resp.\ bottom vertices of the cube chain.

Finally, Lemma \ref{lem:cont} and formula (\ref{eq:t(p)}) show that the construction yields a \emph{continuous} taming map $T:\vec{S}_{\mb{c}}^C(X)_{x^-}^{x^+}\to \vec{ST}_{\mb{c}}(X)_{x^-}^{x^+}$ and a continous deformation $H: \vec{S}_{\mb{c}}^C(X)_{x^-}^{x^+}\to (\vec{S}_{\mb{c}}^C(X)_{x^-}^{x^+})^I$ such that $H_0=id$ and $H_1$ is given by $T$ -- composed with an inclusion map.
Moreover, Remark \ref{rem:taming}(3) shows that  $H$ leaves $\vec{ST}_{\mb{c}}(X)$ elementwise fixed.
\end{proof}

\section{Taming is a homotopy equivalence}\label{s:sdp}
\subsection{Proper and non-self-linked $\Box$-sets}
In this section, we deal with a proper non-self-linked $\Box$-set $X$ (cf Definition \ref{df:pcs}(2)). For a such, the taming result ($\textcircled{3}$ in Theorem \ref{thm}) is a consequence of the nerve lemma. 
This is essentially due to the contractibility of several path spaces from Definition \ref{def:subo}:

\begin{proposition}\label{prop:contr2}
Let $X$ denote a proper non-self-linked $\Box$-set and let $\mb{c}, \mb{c}_1,\dots\mb{c}_k$ denote cube chains in $X$ from $x^-$ to $x^+$.
\begin{enumerate}
\item The spaces $\vec{ST}_{\mb{c}}(X)_{x^-}^{x^+}$ and $\vec{T}_{\mb{c}}(X)_{x^-}^{x^+}$ of (strict) d-paths subordinate to $\mb{c}$ are \emph{contractible}.
\item Likewise the spaces $\vec{S}^C_{\mb{c}}(X)_{x^-}^{x^+}$ and $\vec{P}^C_{\mb{c}}(X)_{x^-}^{x^+}$ of (strict) d-paths subordinate to the collar of $\mb{c}$.
\item The intersections $\bigcap_i\vec{S}^C_{\mb{c}_i}(X)_{x^-}^{x^+}$ and $\bigcap_i\vec{P}^C_{\mb{c}_i}(X)_{x^-}^{x^+}$ are contractible if the cube chains $\mb{c}_i$ possess a common refinement and empty otherwise.
\end{enumerate}
\end{proposition}

\begin{proof}
\begin{enumerate}
\item For $\vec{T}_{\mb{c}}(X)_{x^-}^{x^+}$, this has been observed in \cite[Proposition 6.2(1)]{Ziemianski:17}. In brief, the space of d-paths in a \emph{single} cube from the bottom to the top vertex is contractible (to a constant speed diagonal path joining them, say) by a linear d-homotopy. For paths in a general cube chain, perform first a reparametrization homotopy joining every d-path with its naturalization \cite[Section 2.3]{Raussen:09}, ie, a unit speed path with respect to the  $l_1$-norm along the same trajectory; for more details, we refer to Section \ref{ss:kink}. The space of natural d-paths can then be contracted cubewise.

For $\vec{S}_{\mb{c}}(X)_{x^-}^{x^+}$, one can either go through the same steps for strict d-paths or one can refer to Proposition \ref{prop:strict} (or rather its proof) in the current paper. 
\item For $\vec{S}_{\mb{c}}^C(X)_{x^-}^{x^+}$, this follows from (1) and Proposition \ref{prop:contr}. For $\vec{P}^C_{\mb{c}}(X)_{x^-}^{x^+}$, apply then Proposition \ref{prop:strict} (or rather its proof).
\item follows from (2) and Proposition \ref{prop:refine}(6).
\end{enumerate}
\end{proof}

\begin{theorem}[Nerve lemma]\label{thm:nl}  Let $Z$ denote a paracompact topological space and let $\mathcal{U}$ denote a good covering $Z=\bigcup U_i$ of $Z$ by open sets $U_i\subset X$, ie all non-empty intersections of finitely many sets in $\mathcal{U}$ are contractible. Then $X$ is homotopy equivalent to the nerve of the poset of these non-empty intersections ordered by inclusion.
\end{theorem}

This nerve is a simplicial complex with vertices corresponding to the $U_i$ and $(k-1)$-dimensional simplices corresponding to \emph{non-empty} intersections $\bigcap_{i\in I} U_i$.

It is not completely clear whom to give credit for the nerve lemma originally. It is proved, under a few extra assumptions by Weil in \cite[Section 5]{Weil:52}; Weil refers to Borsuk \cite{Borsuk:46}. The survey paper \cite{Bjorner:95} mentions Leray's \cite{Leray:45} as an even earlier predecessor. A modern formulation and proof  can be found in \cite{McCord:67}. For text book presentations, cf \cite[Corollary 4G.3]{Hatcher:02} or \cite[Theorem 15.21]{Kozlov:08}.\footnote{I would like to thank one of the referees for raising my awareness about the history of the nerve lemma.}
 
\begin{corollary}\label{thm:strict}
\begin{enumerate}
\item The space $\vec{S}(X)_{x^-}^{x^+}$ of strict d-paths is homotopy equivalent to the nerve of a covering $\mc{U}(X)_{x^-}^{x^+}$; likewise,  the space $\vec{ST}(X)_{x^-}^{x^+}$ of strict tame d-paths is homotopy equivalent to the nerve of a covering $\mc{V}(X)_{x^-}^{x^+}$.
\item The two spaces are homotopy equivalent to each other.
\end{enumerate}
\end{corollary}

\begin{proof}
The spaces $\vec{S}_{\mb{c}}^C(X)_{x^-}^{x^+},\; \mb{c}\in\mc{C}(X)_{x^-}^{x^+},$ define a covering $\mc{U}(X)_{x^-}^{x^+}$ of $\vec{S}(X)_{x^-}^{x^+}$ (Proposition \ref{prop:vs}(2)) by open (Proposition \ref{prop:open}(2)) and contractible (Proposition \ref{prop:contr2}) sets. Intersections of sets in the covering are empty or contractible by Proposition \ref{prop:contr2}(3).  Hence the covering $\mc{U}(X)_{x^-}^{x^+}$ is good. Similarly, the spaces $\vec{ST}_{\mb{c}}^C(X)_{x^-}^{x^+},\; \mb{c}\in\mc{C}(X)_{x^-}^{x^+},$ define a good covering $\mc{V}(X)_{x^-}^{x^+}$ of $\vec{ST}_{\mb{c}}^C(X)$. 

Moreover, the spaces $\vec{ST}(X)_{x^-}^{x^+}\subset \vec{S}(X)_{x^-}^{x^+}$ are metrizable and thus paracompact, cf \cite[Corollary 3.1]{Raussen:09}. Apply the nerve lemma, Theorem \ref{thm:nl} to show (1) above.

The two coverings correspond to each other through inclusion maps over the  \emph{same poset} and giving rise to the \emph{same nerve}:  By Proposition \ref{prop:refine}(4), its objects can be enumerated by all sets $\{\mb{c}_i\}$ of cube chains in $X$ from $x^-$ to $x^+$ that possess a common refinement (the partial order is generated by refinement and superset). 
\end{proof}

The remaining paragraphs in this section are not important for the main result, but they lead to a far simpler poset with a smaller nerve, better suited for calculations:  The poset category corresponding to the covering $\mc{U}(X)_{x^-}^{x^+}$ is very redundant. By Proposition \ref{prop:refine}(4), its objects can be enumerated by all sets $\{\mb{c}_i\}$ of cube chains in $X$ from $x^-$ to $x^+$ that possess a common refinement (the partial order is generated by refinement and superset). 

\begin{proposition}\label{prop:nerveC}
The nerves of the covering $\mc{U}(X)_{x^-}^{x^+}$ and of the poset category $\mc{C}(X)_{x^-}^{x^+}$ (cf Definition \ref{def:pcc}(4)) are homotopy equivalent. Hence, the spaces $\vec{ST}(X)_{x^-}^{x^+}\subset\vec{S}(X)_{x^-}^{x^+}$ are both homotopy equivalent to the nerve of $\mc{C}(X)_{x^-}^{x^+}$.
\end{proposition}

\begin{proof}
Let $P\mc{U}(X)_{x^-}^{x^+}$ denote the poset correponding to the covering $\mc{U}(X)_{x^-}^{x^+}$ described in the proof of Corollary \ref{thm:strict} Consider the poset map $P\mc{U}(X)_{x^-}^{x^+}\to\mc{C}(X)_{x^-}^{x^+}$ that associates to a set $\{\mb{c}_i\}$ of cube chains the \emph{coarsest common refinement} of all the $\mb{c}_i$.  The fiber (ie comma category) over any $\mb{c}\in\mc{C}(X)_{x^-}^{x^+}$ has the set $\{\mb{c}\}$ as an initial element, and is thus contractible. Apply Quillen's Theorem A! \cite[page 85]{Quillen:72}, \cite[Theorem 15.28]{Kozlov:08}.
\end{proof}

\subsection{General $\Box$-sets}
I am indebted to K.\ Ziemia\'{n}ski for pointing out to me that Proposition \ref{prop:contr2} is no longer true for non-proper $\Box$-sets.

\begin{example}\label{ex:onecube}
Let $Z_n$ denote the unique $\Box$-set with exactly one cube $c_k,\; 0\le k\le n,$ from Example \ref{ex:cubset}(3). For $n=2$, consider the cube chain $\mb{c}$ consisting of the single cube $c_2$ from $c_0$ to $c_0$. Then $\vec{S}_{\mb{c}}(X)_{c_0}^{c_0}$ is homotopy equivalent to a circle $S^1$ -- and hence not contractible! In fact, it deformation contracts to the subspace of piecewise linear d-paths (cf Section \ref{ss:kink} for details) in $c_2$ connecting $c_0$ with itself through a point on the anti-diagonal (connecting $c_0$ with itself and hence a circle).
\end{example}

For a general $\Box$-set $X$, the proof of Theorem \ref{thm}(1) requires a little more machinery from algebraic topology: Instead of the nerve lemma itself, we apply two more general results in homotopy theory that are used in proving it: the \emph{projection lemma} comparing colimits with homotopy colimits and the \emph{homotopy lemma} comparing homotopy colimits of spaces that can be glued together from pieces that are mutually homotopy equivalent. For a quite elementary presentation, cf eg \cite[Chapter 15]{Kozlov:08}. Using these two results, we can prove that taming is a homotopy equivalence also for general $\Box$-sets:

\begin{proposition}\label{prop:general}
Let $X$ denote a $\Box$-set with selected vertices $x^-, x^+\in X_0$. Then the inclusion map $\iota :\vec{ST}(X)_{x^-}^{x^+}\hookrightarrow \vec{S}(X)_{x^-}^{x^+}$ is a homotopy equivalence.
\end{proposition}

\begin{proof} The proof proceeds by a series of homotopy equivalences (denoted $\simeq$) :\\
$\vec{S}(X)_{x^-}^{x^+}=\colim_{\mc{C}_{x^-}^{x^+}}\vec{S}^C_{\mb{c}}(X)_{x^-}^{x^+}\simeq \hocolim_{\mc{C}_{x^-}^{x^+}}\vec{S}^C_{\mb{c}}(X)_{x^-}^{x^+}\simeq\hocolim_{\mc{C}_{x^-}^{x^+}}\vec{S}_{\mb{c}}(X)_{x^-}^{x^+}$.\\ Likewise, $\vec{ST}(X)=\colim_{\mc{C}_{x^-}^{x^+}}\vec{ST}^C_{\mb{c}}(X)_{x^-}^{x^+}\simeq\hocolim_{\mc{C}_{x^-}^{x^+}}\vec{ST}^C_{\mb{c}}(X)_{x^-}^{x^+}\simeq\hocolim_{\mc{C}_{x^-}^{x^+}}\vec{ST}_{\mb{c}}(X)_{x^-}^{x^+}$.

In both cases, the first homotopy equivalence is due to the projection lemma, the second to the homotopy lemma and Propostion \ref{prop:contr}. Paths subordinate to a cube chain $\mb{c}$ are automatically tame, hence $\vec{S}_{\mb{c}}(X)_{x^-}^{x^+}=\vec{ST}_{\mb{c}}(X)_{x^-}^{x^+}$, ie the last two spaces agree.
\end{proof}

\begin{remark}
By far more sophisticated homotopy theoretical methods, Ziemia\'{n}ski proved in \cite{Ziemianski:19} that $\vec{T}(X)_{x^-}^{x^+}$ is homotopy equivalent to the nerve of a Reedy category $Ch(X)$ instead of our poset category $\mc{C}(X)_{x^-}^{x^+}$, also for general $\Box$-sets. He uses a filtration of $\vec{T}(X)_{x^-}^{x^+}$ by differently defined contractible subsets using tame presentations. But these subspaces do \emph{not} define an open covering, and therefore it is not possible to obtain the result by invoking the nerve lemma! Furthermore,  
Ziemia\'{n}ski shows that $\vec{T}(X)_{x^-}^{x^+}\hookrightarrow\vec{P}(X)_{x^-}^{x^+}$ is a deformation retract by a \emph{global} taming construction that is far more tricky than our local one (ie subordinate to the collar of a particular cube chain).
\end{remark}

\section{pl d-paths and spaces of sequences}\label{ss:kink}
In this final section, we restrict again attention to \emph{proper non-self linked} $\Box$-sets (cf Definition \ref{df:pcs}(2-3)). At least for these, it is possible to find an even smaller model describing the homotopy type of the space of all d-paths between two vertices. To this end, consider the intersections $H_k$ of an $n$-cube $I^n$ (and hence of an $n$-cube in a $\Box$-set $X$) with the hyperplanes  given by the equations $x_1+\dots +x_n=k,\; k\in\mb{N}, 0<k<n$; different from the hyperplanes previously considered. Every hyperplane section $H_k$ is an $(n-1)$-dimensional polyhedron with the vertices with $k$ entries $1$ and $n-k$ entries $0$ as extremal points. Requesting certain variables to take the value $0$ or $1$ yields the restriction of these hyperplane sections to faces; on which they again are hyperplane sections. 

Observe that two elements $x_k\in H_k$ and $x_{k+1}\in H_{k+1}$ such that $x_k\le x_{k+1}$ (ie there exists a d-path from $x_k$ to $x_{k+1}$) have $l_1$-distance (aka.\ Manhattan distance) $d_1(x_k,x_{k+1})=1$.

The hyperplane sections $H_k\subset I^n$ are \emph{achronal}: Each d-path $p$ in $I^n$ intersects a hyperplane section $H_k$ in at most one point $p_k\in H_k$; if defined, $d_1(p_k,p_{k+1})=1$. The union of all hyperplane sections $H_k$ corresponding to all cells in the geometric realization of a $\Box$-set $X$ form a subspace $\mc{H}(X)\subset X$. Any d-path between two elements of $\mc{H}(X)$ has \emph{integral} $l_1$-length (which is thus invariant under directed homotopy). The shortest length is called their $l_1$-distance $d_1$; compare \cite[Section 2.2]{Raussen:09} for this concept in greater generality.

A d-path $p: J\to X$ on an interval $J\subset\mb{R}$ is called \emph{natural} if $d_1(p(t_1),p(t_2))=t_2-t_1$ for $t_1,t_2\in J, t_1\le t_2$. The natural $d$-paths from a vertex $x^-$ to a vertex $x^+$ in $X$ with $p(0)=x^-$ form the space $\vec{N}(X)_{x^-}^{x^+}$. All paths in $\vec{N}(X)_{x^-}^{x^+}$ have integral $d_1$-length. A reparametrization linearly adjusting the domain to length one defines an inclusion map $\iota_N: \vec{N}(X)_{x^-}^{x^+}\hookrightarrow\vec{P}(X)_{x^-}^{x^+}$. Analogous to reparametrization by unit speed for curves in differential geometry, one defines a naturalization map $nat$ in the opposite direction: For a path $p\in\vec{P}(X)_{x^-}^{x^+}$, let $l_p: I\to\mb{R}$ denote the non-decreasing function associating to $t\in I$ the $l_1$-length of the restricted path $p|[0,t]$ and define $nat: \vec{P}(X)_{x^-}^{x^+}\to\vec{N}(X)_{x^-}^{x^+}$ by $nat(p)(s)=p(l_p^{-1}(s))$ (well-defined although $l_p^{-1}(s)$ might be an interval!) This map is homotopy inverse to $\iota_N$:

\begin{proposition}\cite[Proposition 2.15 and Proposition 2.16]{Raussen:09}
For a $\Box$-set $X$ with selected vertices $x^-, x^+\in X_0$, the inclusion map $\iota_N: \vec{N}(X)_{x^-}^{x^+}\hookrightarrow\vec{P}(X)_{x^-}^{x^+}$ is a homotopy equivalence.
\end{proposition}

The natural \emph{tame} d-paths form a subspace with inclusion $\iota_{NT}: \vec{NT}(X)_{x^-}^{x^+}\hookrightarrow\vec{T}(X)_{x^-}^{x^+}$. Naturalization preserves the \emph{trace} of a d-path i.e., its equivalence class up to non-decreasing reparametrization. Hence tame d-paths and tame d-homotopies stay tame under naturalization:

\begin{corollary}\label{cor:nt}
For a $\Box$-set $X$ with selected vertices $x^-, x^+\in X_0$, the inclusion map\\  $\iota_{NT}: \vec{NT}(X)_{x^-}^{x^+}\hookrightarrow\vec{T}(X)_{x^-}^{x^+}$ is a homotopy equivalence.
\end{corollary}

\begin{remark}\label{rem:hypt} Let $p: J\to X$ on an interval with $\min J=0$ denote a natural $d$-path in $X$.
\begin{enumerate}
\item $p$ intersects $\mc{H}(X)$ exactly at \emph{integral} times: $p(t)\in\mc{H}(X)\Leftrightarrow t\in J\cap\mb{Z}$. 
\item  If $p$ is tame, then $p(i)$ and $p(i+1),\; i$ an integer, are contained in a common cube. A minimal such cube is uniquely determined since $X$ is proper. Moreover, since $X$ is non-self-linked, there is a unique unit speed line segment d-path (of length 1) in this (and any other) cube containing them. 
\end{enumerate}
\end{remark}

A path $p\in\vec{NT}(X)_{x^-}^{x^+}$ with $p(0)=x^-$ is called $PL$ (piecewise linear) if, for every integer $i$ in its domain, the path between $p(i)$ and $p(i+1)$ is given by the unit speed line segment (of $l_1$-length 1) in the minimal cube that contains them both. These $PL$ paths between $x^-$ and $x^+$ form the subspace $\vec{PL}(X)_{x^-}^{x^+}\subset\vec{NT}(X)_{x^-}^{x^+}$. 


\begin{proposition}\label{prop:PL}
Inclusion $\iota_{PL}: \vec{PL}(X)_{x^-}^{x^+}\hookrightarrow\vec{NT}(X)_{x^-}^{x^+}\stackrel{\iota_{NT}}{\hookrightarrow}\vec{T}(X)_{x^-}^{x^+}\hookrightarrow\vec{P}(X)_{x^-}^{x^+}$ is a homotopy equivalence.
\end{proposition} 

\begin{proof} 
In view of Theorem \ref{thm} and Corollary \ref{cor:nt}, it remains only to show that the first inclusion is a homotopy equivalence. In fact, $\vec{PL}(X)_{x^-}^{x^+}$ is a deformation retract in $\vec{NT}(X)_{x^-}^{x^+}$: To a natural tame d-path $p: J\to X$ associate
the PL d-path $L(p)$ obtained by linearly connecting $p(i)$ with $p(i+1)$ for $i, i+1\in\mb{Z}\cap J$ in the unique minimal cube containing both, cf Remark \ref{rem:hypt}. Observe that $p=L(p)$ if $p$ is $PL$. The (natural) convex combination homotopy joining $p\in\vec{NT}(X)_{x^-}^{x^+}$ with $L(p)$ shows that the linearization map $L$ thus defined is a homotopy inverse to the first inclusion map. It restricts to the constant homotopy on $\vec{PL}(X)_{x^-}^{x^+}$.
\end{proof}

The only data needed to describe $PL$ d-paths are the \emph{kink points} $p(i)$ in hyperplane sections in the cubes they traverse: \\
The space $Seq(X)_{x^-}^{x^+}$ is  defined as the space of all finite sequences $(x_0=x^-,\dots ,x_n=x^+)$ in $\bigcup_{n\ge 0} \mc{H}(X)^{n+1}$ with $x_i, x_{i+1}$ in a common cube such that $x_i\le x_{i+1}$ and 
$d_1(x_i,x_{i+1})=1$. 

With this definition, we obtain

\begin{proposition}
Let $X$ be a proper non-self-linked $\Box$-set with selected vertices $x^-, x^+\in X_0$. 
\begin{enumerate}
\item $\vec{PL}(X)_{x^-}^{x^+}$ and $Seq(X)_{x^-}^{x^+}$ are homeomorphic.
\item $\vec{P}(X)_{x^-}^{x^+}$ and $Seq(X)_{x^-}^{x^+}$ are homotopy equivalent.
\end{enumerate}
\end{proposition}

\begin{proof} 
\begin{enumerate}
\item The forgetful map that associates to a path $p\in\vec{PL}(X)_{x^-}^{x^+}$ the sequence of kink points 
$(p(i))\in Seq(X)_{x^-}^{x^+}$  -- with $i$ running through the integers in its domain -- is a homeomorphism. Its inverse is the map that associates to a sequence in $Seq(X)_{x^-}^{x^+}$ the $PL$-path that connects any two subsequent elements in that sequence by the unit speed line segment in the unique minimal cube containing them both.
\item follows trom (1) and Proposition \ref{prop:PL}.
\end{enumerate}
\end{proof}

\begin{example}
Let $X=\partial I^3$ denote the boundary of a 3-cube from Section \ref{sss:ex}. The hyperplane sections (diagonal lines) in the six boundary squares form two triangles, cf Figure \ref{fig:hypp}. The associated pairs (= sequences) of kink points between the bottom and the top vertex form a hexagon, homotopy equivalent to the circle $S^1$.
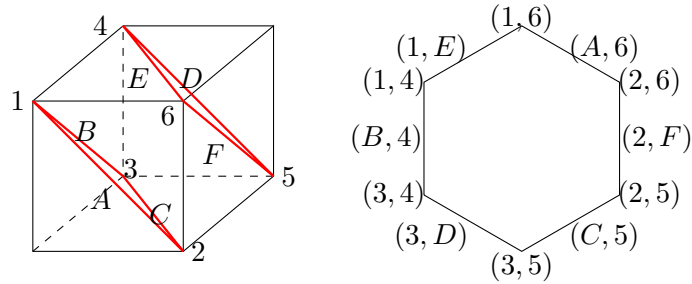
\begin{figure}[h]
\begin{tikzpicture}
\draw (0,0) rectangle (2,2);
\draw (2,0) -- (3.2,1) -- (3.2,3) -- (2,2);
\draw (3.2,3) -- (1.2,3) -- (0,2);
\draw [dashed] (0,0) -- (1.2,1) -- (1.2,3);
\draw [dashed] (1.2,1) -- (3.2,1);
\draw [red, line width= 0.3mm] (0,2) -- (2,0) -- (1.2,1) -- (0,2);
\draw [red, line width= 0.3mm] (2,2) -- (3.2,1) -- (1.2,3) -- (2,2);
\node at (-0.2,2) {$1$};
\node at (2.2,0) {$2$};
\node at (1.3,1.1) {$3$};
\node at (0.9,3) {$4$};
\node at (1.8,1.8) {$6$};
\node at (3.4,1) {$5$};
\node at (0.9,0.7) {$A$};
\node at (1.7,0.5) {$C$};
\node at (0.7,1.6) {$B$};
\node at (2.1,2.3) {$D$};
\node at (2.4,1.3) {$F$};
\node at (1.4,2.3) {$E$};
\draw (5.2,0.75) -- (6.5,0) -- (7.8,0.75) -- (7.8,2.25) -- (6.5,3) -- (5.2,2.25) -- (5.2,0.75);
\node at (4.8,2.25) {$(1,4)$};
\node at (6.5,3.1) {$(1,6)$};
\node at (8.2,2.25) {$(2,6)$};
\node at (8.2,0.75) {$(2,5)$};
\node at (6.5,-0.2) {$(3,5)$};
\node at (4.8,0.75) {$(3,4)$};
\node at (5.3,2.7) {$(1,E)$};
\node at (4.7,1.5) {$(B,4)$};
\node at (5.3,0.2) {$(3,D)$};
\node at (7.6,0.2) {$(C,5)$};
\node at (8.3,1.5) {$(2,F)$};
\node at (7.6,2.7) {$(A,6)$};
\end{tikzpicture}
\caption{Hyperplane sections (=red lines) in the boundary $\partial I^3$ of a 3-cube and associated pairs of kink points. For example $(A,6)$ indicates that from a point on the interior of $A$ you can reach only $6$ by a tame path; from its end points $1$, resp.\ $2$ you can reach $E$, resp.\ $F$.}\label{fig:hypp}
\end{figure}

\end{example}

\begin{remark}
Extending the results of this section to a general $\Box$-set $X$ seems to be more intricate. The main reason is that two elements in successive hyperplane sections may be joined by more than one unit speed line segment paths -- through different cubes if $X$ is not proper and/or through the same cube if $X$ is self-linked. For example, two vertices in subsequent hyperplane sections may be connected by various edges, after identification of vertices.

It seems to be necessary to replace the cube chains from this paper by the cube chains $Ch(X)$ in Ziemia\'{n}ski's paper \cite{Ziemianski:19}; those are generated by cubical \emph{maps} from a wedge of cubes into $X$. Hyperplane sections in $X$ should then be replaced by hyperplane sections in a wedge of cubes. In such a wedge of cubes (which is obviously both proper and non-self-linked), there is again a well-defined unit speed line segment between points on consecutive hyperplane sections.
\end{remark}

Competing interests: The author declares none.

\end{document}